\documentclass[a4paper, 11pt]{article}
\usepackage[english]{babel}
\usepackage{amsmath, amssymb, amsfonts, amsthm, bm}
\usepackage{mathrsfs}
\usepackage[mathcal]{euscript}
\usepackage{enumerate}
\usepackage{color,graphicx,epstopdf}
\usepackage{dsfont}
\usepackage{hyperref}
\usepackage{cite} % must be loaded after hyperref
\usepackage{algorithm} %format of the algorithm
\usepackage{algorithmic} %format of the algorithm
\usepackage{tikz-cd}
\usepackage{mathtools}

\newcommand{\R}{\mathbb{R}}

\newcommand{\1}{\mathbf{1}}

\DeclareMathOperator{\rank}{rank}

\newcommand{\btoi}[1]{\ensuremath{\lfloor #1\rfloor}}
\newcommand{\itob}[1]{\ensuremath{\lceil #1\rceil}}

\DeclareMathOperator*{\argmin}{argmin}

\newcommand{\vb}{\mathbf{v}}

\newcommand{\xb}{\mathbf{x}}
\newcommand{\yb}{\mathbf{y}}
\newcommand{\zb}{\mathbf{z}}

\newcommand{\Ib}{\mathbf{I}}

\newcommand{\omegab}{\bm{\omega}}

\DeclareFontFamily{OT1}{pzc}{}
\DeclareFontShape{OT1}{pzc}{m}{it}{<-> s * [1.200] pzcmi7t}{}
\DeclareMathAlphabet{\mathpzc}{OT1}{pzc}{m}{it}

\newcommand*\mcap{\mathbin{\mathpalette\mcapinn\relax}}
\newcommand*\mcapinn[2]{\vcenter{\hbox{$\mathsurround=0pt
  \ifx\displaystyle#1\textstyle\else#1\fi\bigcap$}}}

\newcommand*\mcupinn[2]{\vcenter{\hbox{$\mathsurround=0pt
  \ifx\displaystyle#1\textstyle\else#1\fi\bigcup$}}}

\newtheorem{theorem}{Theorem}
\newtheorem{definition}{Definition}

\newtheorem{lemma}{Lemma}
\newtheorem{remark}{Remark}

\title{\bf Distributed Algorithms that Solve Boolean Equations\\ with Local and Differential Privacies\thanks{A preliminary version of the work was reported at the IEEE Conference on Decision and Control in 2020 \cite{cdc}. }}
\date{}

 	\author{Hongsheng Qi\thanks{Key Laboratory of Systems and Control, Institute of Systems Science, Academy of Mathematics and Systems Science, Chinese Academy of Sciences, Beijing 100190, China; School of Mathematical Sciences, University of Chinese Academy of Sciences, Beijing 100049, China. (qihongsh@amss.ac.cn)}, Bo Li\thanks{Key Laboratory of Mathematics Mechanization, Academy of Mathematics and Systems Science, Chinese Academy of Sciences, Beijing 100190, China. (libo@amss.ac.cn)}, Rui-juan Jing\thanks{Faculty of Science, Jiangsu University, Zhenjiang 212013,  China. (rjing@ujs.edu.cn)}, Lei Wang\thanks{Australian Center for Field Robotics, School of Aerospace, Mechanical and Mechatronic Engineering, The University of Sydney, NSW 2006, Australia. (lei.wang2@sydney.edu.au)}, Alexandre Proutiere\thanks{Department of Automatic Control, KTH Royal Institute of Technology, Stockholm 100-44, Sweden. (alepro@kth.se)}, Guodong Shi\thanks{Australian Center for Field Robotics, School of Aerospace, Mechanical and Mechatronic Engineering, The University of Sydney, NSW 2006, Australia. (guodong.shi@sydney.edu.au)}}
\begin{document}
\maketitle

\begin{abstract}
 In this paper, we propose  distributed algorithms that  solve a system of Boolean equations over a network, where each node in the network possesses only  one Boolean equation from the system. The Boolean equation assigned at any particular  node is a {\em private} equation known to this node only, and the nodes aim to compute the exact set of solutions to the system without exchanging their local equations. We show that each private Boolean equation can be locally lifted to a linear  algebraic equation under a basis of Boolean vectors, leading to a network linear equation that is distributedly solvable using existing distributed linear equation algorithms as a subroutine. A number of exact or approximate solutions to the induced linear equation are then computed at each node from different initial values. The solutions to the original  Boolean equations are eventually computed locally via a Boolean vector search algorithm. We prove that given solvable Boolean equations,  when the initial values of the nodes for the distributed linear equation solving step are i.i.d selected according to a uniform distribution in a high-dimensional cube,   our   algorithms return the exact solution set of the Boolean equations at each node with high probability. Furthermore,  we present  an algorithm for distributed verification of the satisfiability of Boolean equations, and prove its correctness. Finally, we show that by utilizing   linear equation solvers with differential privacy  to replace the in-network computing routines, the overall distributed Boolean equation algorithms can be made differentially private. Under the standard Laplace mechanism, we prove an explicit level of noises that can be injected in the linear equation steps for ensuring  a prescribed level of differential privacy.

 \end{abstract}

 \section{Introduction}

 Computing the solutions to a system of Boolean equations is a fundamental computation problem. The Boolean satisfiability ({\sf SAT}) problem for determining whether a Boolean formula is satisfiable or not, was the first computation problem proven to be NP-complete \cite{complexitybook}. The solvability of (static) Boolean equations is directly related to problems for gene regulations in biology, and   for security of the keystream
generation in cryptography \cite{corblin2007,bardet2013}. In the meantime,  Boolean dynamical systems have found broad  applications in  modeling epidemic processes of computer networks \cite{virus},    social opinion dynamics \cite{social},   and decision making in economics \cite{eco} since logical states are ubiquitous in our world. For such systems, finding their steady states, if any, again falls to a problem of solving a system of Boolean equations \cite{ara1,ara2}. As a result, static or dynamic Boolean equations attracted much research interest in computing theory and engineering  in terms of complexity analysis,  efficient algorithm design, and signal processing \cite{bardet2013,Chaves2013,dimitri2019}.

We are interested in systems of Boolean equations that are defined over a network, and aim to develop distributed algorithms that can solve such equations, and verify their satisfiability in a distributed manner over the network.

\subsection{System of Boolean Equations over a Network}
Consider the following system of Boolean equations with respect to decision variable $x=(x_1,\dots,x_m)$:
\begin{equation}
\begin{aligned} \label{bes}
&f_1(x_1,\dots,x_m)=\sigma_1 \\
&\ \ \ \vdots \\
&f_n(x_1,\dots,x_m)=\sigma_n
\end{aligned}
\end{equation}
where $x_i\in\{0,1\}$ for $i=1,\dots,m$ and $\sigma_j\in\{0,1\}$, each $f_j$ maps from the $m$-dimensional binary space $\{0,1\}^m$ to the binary space $\{0,1\}$ for $j=1,\dots,n$. The solution set to the $i$th equation in the system
$
f_i(x)=\sigma_i
$,
is denoted as $\mathcal{S}_i\in \{0,1\}^m$.

Let there be a network with $n$ nodes indexed in the set $\mathrm{V}=\{1,\dots,n\}$. The  communication structure of the network is described by a simple, undirected, and connected graph $\mathrm{G}=(\mathrm{V},\mathrm{E})$, where each edge  $\{i,j\}\in\mathrm{E}$ is an un-ordered pair of two distinct nodes in the set $\mathrm{V}$. The neighbor set of node $i\in\mathrm{V}$ over the network is specified as $\mathrm{N}_i=\{j:\{i,j\}\in\mathrm{E}\}$.   The $i$th equation in the Boolean equation system (\ref{bes}), along with its solution set $\mathcal{S}_i$, is assumed to be a local and private equation assigned to node $i$ only. We are interested in the following distributed computation problem, where the nodes aim to compute the solutions to the Boolean equation system (\ref{bes}) over the network $\mathrm{G}$ distributedly.

\begin{definition}\label{def:algo.dist}(Locally Private Distributed Algorithm)
A distributed algorithm that  solves  the Boolean equation system (\ref{bes})  satisfies
\begin{itemize}
\item[(i)] (Local privacy) Nodes do not  directly reveal their private Boolean equations or their solution sets to other nodes.
\item[(ii)] (Decentralized computation) Each node $i\in\mathrm{V}$ holds   a dynamical state $\mathbf{x}_i(t)$, shares the state  with her neighbors, and carries out local computations for the update of  $\mathbf{x}_i(t)$ based on  her private Boolean equation and received neighbor states.
\item[(iii)] (Output consistency) The solution set of the Boolean equation system (\ref{bes}) is   obtained at all nodes as the algorithm output.
\end{itemize}
\end{definition}

This notion of locally private distributed algorithm is consistent with the line of study  on distributed convex optimization  \cite{nedic2010constrainedTAC}, where nodes in a network hold local cost functions and the goal of the network is to minimize the summation of all cost functions with a common decision variable among the nodes. Note that, since the representation of the Boolean formula $f_i$ is not unique,   it is reasonable to define the local privacy on the solution set $\mathcal{S}_i$.  In practice, the local privacy is motived from the fact that nodes do not necessarily trust their neighbors \cite{tomlin20}. We note that the shared dynamical states $\mathbf{x}_i(t)$ will certainly have to  contain information about the Boolean equation $f_i(x)=\sigma_i$ and its solution set $\mathcal{S}_i$. Neighbors of node $i$, or eavesdroppers having access to  $\mathbf{x}_i(t)$ might infer the  equation  $f_i(x)=\sigma_i$ or the solution set $\mathcal{S}_i$. Therefore, Definition \ref{def:algo.dist}  does not capture this  {\it indirect} privacy risk.

{In practice, it is also of interest to have a quantitative measure for distributed algorithms regarding the privacy risks of the private datasets when the network may face eavesdroppers having access to  all node-to-node communications, e.g., \cite{d3}.  One popular privacy metric has been differential privacy, where in a probabilistic setting,  privacy of the datasets is measured by the regularity of the likelihood function of the node-to-node communications with respect to the private data  \cite{dwork2014algorithmic}. Thus in differentially private algorithms, there must be certain randomization mechanisms.   } Let $f(x)=\sigma$ and $f'(x)=\sigma$ be two Boolean equation systems in the form of (\ref{bes}). Denote  $\mathcal{S}=\mathcal{S}_1\times \dots \times \mathcal{S}_n,\mathcal{S}'=\mathcal{S}_1'\times \dots \times \mathcal{S}_n'$ as their (local) solution sets, respectively. Then we term the two Boolean equations   to be adjacent if there is a unique $k\in\{1,\dots,n\}$ such that $\mathcal{S}_k$ and $\mathcal{S}_k'$  with $|\mathcal{S}_k|=|\mathcal{S}_k'|$ differ by only one entry, and the rest $\mathcal{S}_j$ and $\mathcal{S}_j'$
are identical.  Let $\mathrm{M}$ be the mapping from  $\mathcal{S}$ to    the node states $x_i(t)$ for all $i$ and for all $t$.  We further introduce the following definition on differential privacy of a distributed algorithm for the Boolean equations.

\begin{definition}\label{dp} (Differentially Private Distributed Algorithm)
Consider a (randomized) distributed algorithm  for the Boolean equation system (\ref{bes}). Let $\mathbb{P}$ denote the probability measure that captures all randomness in the algorithm.   Then the algorithm  is called  $\epsilon$-differentially private  if there holds
\begin{align}
\mathbb{P}(\mathrm{M}(\mathcal{S})\in \mathcal{E})\leq e^\epsilon  \mathbb{P}(\mathrm{M}(\mathcal{S}')\in \mathcal{E})
\end{align}
for any pair of two adjacent   $\mathcal{S}$ and $\mathcal{S}'$ and for any event $\mathcal{E}$ of the network node states.
\end{definition}

\subsection{Contributions}
First of all,  a distributed algorithm in the sense of Definition \ref{def:algo.dist} is  developed  that  solves  a system of Boolean equations that is satisfiable over a network.
This algorithm consists of three ingredients:
\begin{itemize}
\item[(i)] Each private Boolean equation is lifted by local computation at each node to a linear  algebraic equation under a basis of Boolean vectors;
\item[(ii)] The network nodes distributedly  compute a number of exact solutions to the induced system of linear equations via an exact projection consensus algorithm \cite{nedic2010constrainedTAC}  from randomly selected  initial values;
\item[(iii)] Each node locally computes the solution set of original system of Boolean equations from the linear equation solutions by a Boolean vector search algorithm.
\end{itemize}
We prove that if the initial values of the nodes for the distributed linear equation solving step are i.i.d generated according to a uniform distribution in a high-dimensional cube, with probability one the proposed algorithm returns the exact solution set of the Boolean equations at each node throughout the network. The complexity of the local Boolean vector search algorithm is also established. Next, we generalize this algorithm based on exact projection consensus, which in theory would require an infinite number of iterations,  to an algorithm that relies only on approximate projections with a finite number of iterations. We also prove this extended algorithm can with high probability guarantee  the exact solution set of the Boolean equations at each node to be computed at each node. Finally, we propose an algorithm for distributed verification of the satisfiability of the system of Boolean equations, and prove its correctness. In all these algorithms, if  the subroutines for solving the induced   linear equations can be made differentially private and computationally accurate (see e.g. \cite{PPSC}), the distributed Boolean equation algorithms become differentially private.  Under standard Laplace mechanism and the projection consensus algorithm, we prove an explicit level of noises to be injected in the linear equation steps for ensuring a prescribed level of differential privacy. Some preliminary results on the algorithms for solving the Boolean equation system with known solvability  were reported at the IEEE Conference on Decision and Control in 2020 \cite{cdc}.

\subsection{Related Work}
Our  definition of distributed algorithms for  Boolean equation systems is along the same line of research on  distributed convex optimization  \cite{nedic2010constrainedTAC}, network linear  equation solving \cite{mou2015,shi2017networkTAC}, distributed signal processing \cite{2010}, distributed sensor estimation \cite{kar2012tit}, and distributed stochastic approximation over networks \cite{bianchi2013tit}. This line of studies can be traced back to the seminal work of Tsitsiklis, Bertsekas, and   Athans in 1986 \cite{tsi1986}. The central idea is, for a global computation task decomposed over a number of nodes in a network,  the task can be achieved at each node by the nodes  running local computations according to their local tasks which are interconnected by an average consensus process.  Since average consensus processes are inherently distributed relying only on local and anonymous communications, and rather robust against  factors such as disturbances, communication failures, network structural changes, etc., \cite{kar2010,tit2010,shi2015tit}, such  distributed computations have the advantage of being resilient and scalable. In fact, the construction of our algorithms for solving the Boolean equations distributedly  utilizes the framework on distributed linear equation solving  \cite{mou2015,shi2017networkTAC} as a subroutine, and the node communications in our algorithms do not rely on node identities which enables anonymous computing  \cite{julien}.

An intuitive distributed algorithm that solves the considered Boolean equation systems may be established via  recursive set operations. In particular, one can set $S_i(0)$ as the solution set $\mathcal{S}_i$, and let $S_i(t)$ be a sequence of subset of $\mathcal{S}_i $ from the recursion 
\begin{align}\label{set}
S_i(t+1)=\bigcap_{j\in \mathrm{N}_i} S_j(t), \ i=1,\dots,n.
\end{align}
In this case, each $\mathcal{S}_i$ can be encoded into a vector with $2^m$ dimensions, and the recursion indeed computes the solution set to the Boolean equation (\ref{bes}) in $D_{\mathrm{G}}$ steps, where $D_{\mathrm{G}}$ is the diameter of the graph $\mathrm{G}$. On the algorithm (\ref{set}), there have been established results  on tight lower and upper bounds for both   communication complexity and   steps needed for convergence   \cite{siam1992,focs,peleg}.  We note that in the recursion (\ref{set}), each node needs to send its solution set $\mathcal{S}_i $ to its neighbors, leading to a violation of  the local privacy condition in Definition 1. Therefore, our approach taking the in-network computations in a $2^m$-dimensional Euclidean space avoids this direct local privacy risk, and incorperating differentially private  linear algebraic equation solvers can even make the overall algorithms  differentially private in a global sense. In the meantime, how to modify (\ref{set}) so that it becomes locally and differentially private seems not to be a straightforward  problem.

We also note that our work is related, but quite different from the work on parallel algorithms for SAT problems \cite{psat1,psat2,psat3}. The focus on parallel SAT solvers is about using a multicore architecture from a set of processing units  with  a shared memory to solve SAT problems with specific structures \cite{psat1,psat2,psat3}. In comparison with the distributed computation setting, the parallel decomposition of the problem is part of the design of the algorithm, and the shard memory implies an all-to-all communication structure among the processing units. Therefore, the primary motivation of the proposed  distributed algorithms for Boolean equations is not acceleration of  centralized algorithms, but rather exploring the possibilities of generalizing the distributed computing framework for continuous and convex problems to  Boolean equations with a discrete and combinatorial nature.

\subsection{Paper Organization}
In Section \ref{secpre}, we introduce some preliminary algorithms and results for Boolean equations, average consensus algorithms, and projection consensus algorithms. Section \ref{secsolver} presents the algorithms that are used for distributedly solving a system of Boolean equations, and prove their correctness provided that the equations are satisfiable. Then Section \ref{secsat} presents a distributed  algorithm  that can verify satisfiability definitively  for a system of Boolean equations over a network. Finally in Section \ref{seccon} some concluding remarks are given.

\section{Preliminaries}\label{secpre}
We first present some preliminaries on the concepts and tools that are incremental for the construction of our algorithm.  First of all, we  review methods that can generate  a matrix representation for a Boolean mapping when logical variables are represented by certain Boolean vectors (${\sf BooleanMatricization}$), and algorithms that can solve a system of linear algebraic equations distributedly over a network ({\sf DistributedLAE}). Although such algorithms are well established in their respective literature, we present some basic form and analysis of the algorithms in order to facilitate  a self-contained presentation. Then, we  present a novel search algorithm ({\sf BooleanVectorSearch}) for localizing  all Boolean vectors in a given affine subspace, and prove its correctness and computational efficiency.

\subsection{Matrix Representation of Boolean Mappings}

For any integer  $m\geq 2$,  we introduce two mappings:
\begin{itemize}
\item[(i)] $\btoi{\cdot}:\{0,1\}^m\rightarrow\{1,\dots, 2^m\}$, where $\btoi{i_1\cdots i_m}=\sum_{k=1}^m i_k2^{m-k}+1$;

\item[(ii)]
$\itob{\cdot}:\{1,\dots, 2^m\}\rightarrow\{0,1\}^m$ with $\itob{i}=[i_1\dots i_m]$ satisfying $i=\sum_{k=1}^m i_k2^{m-k}+1$.
\end{itemize}

For any integer $c$, we let $\delta_c^i$ be the $i$-th column of the $c\times c$ identify matrix $\mathbf{I}_c$, and define  $$
\Delta_c=\{\delta_c^i:\ i=1,\dots,c\}.
$$
We can then establish  a one-to-one correspondence between the elements in $\{0,1\}^m$ and the elements in $\Delta_{2^m}$. To this end, we define
\begin{itemize}
\item  The mapping $\Theta_m(\cdot)$ from  $\{0,1\}^m$ to $\Delta_{2^m}$:
\begin{align}\label{eq:composite.boolean.map}
\Theta_m(x):=\delta_{2}^{x_1+1}\otimes\cdots\otimes\delta_{2}^{x_m+1}=\delta_{2^m}^{\sum_{i=1}^m x_i2^{m-i}+1}=\delta_{2^m}^{\lfloor  x \rfloor},
\end{align}
for all    ${x}=[x_1,\dots,x_m]\in \{0,1\}^m$, where $\otimes$ represents the Kronecker product.
\item The mapping  $\Upsilon_m(\cdot)$ from  $\Delta_{2^m}$ to $\{0,1\}^m$:
\begin{align}
	\Upsilon_m(\delta_{2^m}^i):= \itob{i}
\end{align}
for all $i=1,\dots,2^m$.
\end{itemize}
We can easily verify
$
\Upsilon_m\big(\Theta_m(x)\big)=x
$ holds for all $x\in \{0,1\}^m$, and therefore, we have obtained a desired bijective mapping   between  $\{0,1\}^m$ and $\Delta_{2^m}$. It turns out, with the help of the vectors in $\Delta_{2^m}$, one can represent any Boolean mapping  $g(x_1,\dots,x_m)$ from $\{0,1\}^m$ to $\{0,1\}$ by a matrix $\mathbf{M}_g\in \mathbb{R}^{2\times 2^m}$.

\begin{lemma}\label{prop1}
Let $g(x_1,\dots,x_m)$ be a Boolean mapping from $\{0,1\}^m$ to $\{0,1\}$. Then there exists $\mathbf{M}_g\in \mathbb{R}^{2\times 2^m}$ as a representation of $g(\cdot)$ in the sense that
\begin{align}
\mathbf{M}_g \big(\Theta_m(x)\big) = \Theta_1 \big(g(x)\big)
\end{align}
for all  $x=(x_1,\dots,x_m)\in\{0,1\}^m$.
\end{lemma}

Lemma \ref{prop1} can be viewed as a special case of representing Boolean mappings by matrices in \cite{Cheng2009} for the analysis of  Boolean dynamical  networks. However, the existence of the matrix $\mathbf{M}_g$ can in fact be directly established from the following commutative  diagram  in view of the identity  $
\Upsilon_m\big(\Theta_m(x)\big)=x
$:
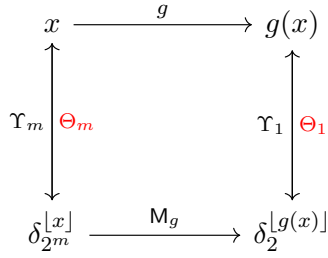
\begin{figure}[H]
\centering
\tikzcdset{row sep/normal=2cm,column sep/normal=2cm}
\begin{tikzcd}
x \arrow[r, "g"] \arrow[d, "\Theta_m" red]
& g(x) \arrow[d, "\Theta_1" red] \\
\delta_{2^m}^{\lfloor  x \rfloor} \arrow[r, "\mathsf{M}_g" black] \arrow[u, "\Upsilon_m"]
& \delta_{2}^{\lfloor  g(x) \rfloor} \arrow[u, "\Upsilon_1"]
\end{tikzcd}
\caption{The commutative digram for mappings over the logical spaces and the vector spaces. }
\end{figure}
Here in the diagram  $\mathsf{M}_g$ is a mapping  from $\Delta_{2^m}$ to $\Delta_{2}$.  Note that the elements in $\Delta_{2^m}$ (resp. $\Delta_{2}$) form a basis of $\mathbb{R}^{2^m}$ (resp. $\mathbb{R}^2$),  $\mathsf{M}_g$ can be naturally understood as a linear operator from  $\mathbb{R}^{2^m}$ to  $\mathbb{R}^2$. As a result, the matrix $\mathbf{M}_g$ can be obtained directly as a matrix representation of $\mathsf{M}_g$ by
\begin{align}\label{eq1}
\mathbf{M}_g=\Big[\delta_{2}^{g(\itob{1})}\ \delta_{2}^{g(\itob{2})}\ \dots, \ \delta_{2}^{g(\itob{2^m})}\Big]
\end{align}
where the $i$th column of $\mathbf{M}_g$ is the $\Theta_1\big(g(\itob{i})\big)$, i.e., the coordinate of $\mathsf{M}_g (\delta_{2^m}^{i})$ under the basis $\Delta_1$. From this perspective we also see that the matrix $\mathbf{M}_g$ stated in Lemma \ref{prop1} is in fact unique.

The equation (\ref{eq1}) serves also as a brutal-force way of computing the matrix $\mathbf{M}_g$ for any  $g(x_1,\dots,x_m)$ from $\{0,1\}^m$ to $\{0,1\}$, where one needs to examine the value of $g$ for all $x\in\{0,1\}^m$. A  systematic  way of computing the matrix $\mathbf{M}_g$ from the original Boolean formula  of $g$ is the {\em semi-tensor product approach} developed by Cheng and his colleagues \cite{Cheng2009,Cheng-Qi2010}, which puts the study of Boolean dynamical networks into a fully algebraic framework. We refer to \cite{book2011} for a detailed  introduction to semi-tensor product approach and its applications to Boolean networks. %Therefore, we can conclude that there are existing methods for computing  $\mathbf{M}_g$ from $g$. We denote by ${\sf BooleanMatricization}$ an algorithm that produces a matrix representation as output from a Boolean function as  input.
\begin{remark}
Lemma \ref{prop1} can be easily  generalized to Boolean mappings $g(x_1,\dots,x_m)$   from $\{0,1\}^m$ to $\{0,1\}^l$, where the resulting $\mathbf{M}_g$ becomes a $2^{l}\times 2^n$ matrix. The existence of such a matrix follows from the same argument, which implies
$$
\mathbf{M}_g=\Big[\delta_{2^l}^{\btoi{g(\itob{1})}}\ \delta_{2^l}^{\btoi{g(\itob{2})}}\ \dots, \ \delta_{2^l}^{\btoi{g(\itob{2^m})}}\Big].
$$
Again, by the semi-tensor product approach, one can alternatively compute this matrix representation for any given Boolean formula   \cite{book2011}.
\end{remark}
We denote by ${\sf BooleanMatricization}$ an algorithm that produces a matrix representation as output from a Boolean mapping  as  input, e.g., $
\mathbf{M}_g={\sf BooleanMatricization}(g)
$ for any  Boolean mapping $g$ from $\{0,1\}^m$ to $\{0,1\}^l$.

\subsection{Projection Consensus for Linear Algebraic Equations}

Consider the network $\mathrm{G}=(\mathrm{V},\mathrm{E})$. Suppose each node $i\in\mathrm{V}$ holds $(\mathbf{H}_i,\mathbf{z}_i)$ with $\mathbf{H}_i \in\mathbb{R}^{k\times d}$ and $\mathbf{z}_i\in \mathbb{R}^k$ defining a linear equation
\begin{align}\label{eqlinear}
  \quad \mathbf{H}_i \mathbf{y}= \mathbf{z}_i
\end{align}with respect to an unknown $\mathbf{y}\in\R^d$. Then over the network  the following  system of linear algebraic equations is defined
\begin{equation}\label{LAE}
\mathbf{H} \mathbf{y}=\mathbf{z},
\end{equation}
where
\[
\mathbf{H}=\begin{bmatrix}
\mathbf{H}_1 \\
\vdots\\
\mathbf{H}_n
\end{bmatrix} \in\mathbb{R}^{(nk)\times d},\ \ \mathbf{z}=\begin{bmatrix}
\mathbf{z}_1 \\
\vdots\\
\mathbf{z}_n
\end{bmatrix}\in \mathbb{R}^{nk}.
\]
Define the following affine subspace in  $\mathbb{R}^d$ specified by the linear equation (\ref{eqlinear}):
$$
\mathcal{E}_i:=\big\{\mathbf{y}\in \mathbb{R}^d:\  \mathbf{H}_i \mathbf{y}= \mathbf{z}_i \big\}
$$
and then let $\mathcal{P}_i(\cdot)$ be the projector onto the affine space $\mathcal{E}_i$.  Let $\mathcal{P}_\ast(\cdot)$ be the projector onto the affine subspace $$
\mathcal{E}:=\big\{\mathbf{y}\in \mathbb{R}^d:\  \mathbf{H} \mathbf{y}= \mathbf{z} \big\}
$$
whenever $\mathcal{E}=\bigcap_{i=1}^n \mathcal{E}_i$ is nonempty.  We can utilize these projectors to obtain a distributed projection consensus solver to the system of linear equations (\ref{LAE}).
\begin{lemma}\label{prop2}
Suppose the system of linear equations (\ref{LAE}) admits at least one exact solutions. Then along the recursion
\begin{align}\label{eq2}
\mathbf{x}_i(t+1)=\mathcal{P}_i\Big(\mathbf{x}_i(t)+\epsilon \sum_{j\in\mathrm{N}_i}\big(\mathbf{x}_j(t)- \mathbf{x}_i(t)\big)\Big), \ \ i=1,\dots,n
\end{align}
there holds for   $0<\epsilon<1/n$ that
$\lim_{t\to \infty}\mathbf{x}_i(t)=\sum_{k=1}^n\mathcal{P}_\ast (\mathbf{x}_k(0))/n \in\mathcal{E}
$ for all $i\in\mathrm{V}$.
\end{lemma}

The projection consensus algorithm (\ref{eq2}) is in fact a special case of the constrained consensus algorithm presented in  \cite{nedic2010constrainedTAC} for distributedly computing a point in the intersection of several convex sets, and   the convergence statement in Lemma \ref{prop2} follows from Lemma 3 of \cite{nedic2010constrainedTAC}. This algorithm (\ref{eq2}) is a form of alternating projection algorithms for convex feasibility problems \cite{apa}. The rate of convergence is in fact exponential.
The specific representation of the node state  limit can be established via Lemma 5 in \cite{shi2017networkTAC}, based on which we can obtain
\begin{align}\label{3}
\sum_{i=1}^n\mathcal{P}_\ast(\mathbf{x}_i(t+1))&=\sum_{i=1}^n\mathcal{P}_\ast\Big(\mathbf{x}_i(t)+\epsilon \sum_{j\in\mathrm{N}_i}\big(\mathbf{x}_j(t)- \mathbf{x}_i(t)\big)\Big)\nonumber\\
&=\sum_{i=1}^n \sum_{j=1}^n W_{ij} \mathcal{P}_\ast (\mathbf{x}_j(t))\nonumber\\
&=\sum_{i=1}^n\mathcal{P}_\ast(\mathbf{x}_i(t))
\end{align}
for all $t\geq 0$ along (\ref{eq2}). Here $W_{ij}= \epsilon$ for all $\{i,j\}\in\mathrm{E}$, $W_{ii}= 1-|\mathrm{N}_i|\epsilon$ for all $i\in\mathrm{V}$, and $W_{ij}=0$ otherwise. Invoking Lemma 3 of \cite{nedic2010constrainedTAC}, we know that all $\mathbf{x}_i(t)$ converge to a common limit value in $\mathcal{E}$, which has to be $\sum_{k=1}^n\mathcal{P}_\ast (\mathbf{x}_k(0))/n$ from (\ref{3}).

\begin{remark}
Without the projection in (\ref{eq2}), the algorithm
\begin{align}\label{consensus}
\mathbf{x}_i(t+1)=\mathbf{x}_i(t)+\epsilon \sum_{j\in\mathrm{N}_i}\big(\mathbf{x}_j(t)- \mathbf{x}_i(t)\big)
\end{align}
is a standard average consensus algorithm \cite{magnusbook}. If $0<\epsilon<1/n$, then along (\ref{consensus}) there holds
$\lim_{t\to\infty}x_i(t)= \sum_{j=1}^n x_j(0)/n$ for all $i=1,\dots,n$.
\end{remark}Besides the algorithm (\ref{eq2}),  there are various other distributed algorithms that produce a solution to (\ref{LAE}), e.g.,   \cite{mou2015,shi2017networkTAC}. We denote by ${\sf DistributedLAE}$ an algorithm that solves a linear algebraic  equation  (\ref{LAE}) with exact solutions  distributedly over a graph $\mathrm{G}$, and formally we write
$$
\mathbf{y}_\ast (\mathbf{x}(0))={\sf DistributedLAE}\big(\mathcal{E}_i, \mathbf{x}_i(0):i\in \mathrm{V}\big)=\sum_{k=1}^n\mathcal{P}_\ast (\mathbf{x}_k(0))/n
$$

\subsection{Boolean Vector Search in an Affine Subspace}

We term the vectors in $\Delta_{2^m}$ as {\it Boolean vectors} in $\mathbb{R}^{2^m}$ as noted from their close relationship with the logical states in $\{0,1\}^m$. In this subsection, we provide an algorithm to search the Boolean vectors in any affine subspace of $\mathbb{R}^{2^m}$.

Let $\yb_1,\dots,\yb_k \in \mathbb{R}^{2^m}$.  Let ${\sf
  SubSpace}(\yb_1, \ldots, \yb_k)$ be the subspace generated by
$\yb_1, \ldots, \yb_k$.  By definition, ${\sf SubSpace}(\yb_1, \ldots,
\yb_k) := \{ \alpha_1 \yb_1 + \cdots + \alpha_k \yb_k \ | \ \alpha_1,
\ldots, \alpha_k \in \R \}$.   Let ${\sf Aff}(\yb_1,\dots,\yb_k)$ be the generated affine subspace
from $\yb_1,\dots,\yb_k$ in $\mathbb{R}^{2^m}$, as the minimal affine
subspace that contains all $\yb_i$.
Geometrically, the affine subspace ${\sf Aff}(\yb_1, \ldots, \yb_k)$
can be obtained by translating ${\sf SubSpace}(\yb_2 - \yb_1, \ldots,
\yb_k - \yb_1)$ from the origin to the point $\yb_1$.  Therefore, a
vector $\yb$ is contained in ${\sf Aff}(\yb_1, \ldots, \yb_k)$ if and
only if $\yb_1 - \yb$ is contained in ${\sf SubSpace}(\yb_2 - \yb_1,
\ldots, \yb_k - \yb_1)$.
The dimension of the affine subspace ${\sf Aff}(\yb_1, \ldots, \yb_k)$
is defined as the dimension of the subspace ${\sf SubSpace}(\yb_2 -
\yb_1, \ldots, \yb_k- \yb_1)$.  We are interested in the following Boolean vector search problem.

\medskip

\noindent{\bf Boolean Vector Search.}  Suppose
%${\rm dim}({\sf Aff}(\yb_1,\ldots,\yb_k))= b$ and
${\sf   Aff}(\yb_1,\dots,\yb_k)\bigcap\Delta_{2^m}\neq \emptyset$. Find all vectors in ${\sf   Aff}(\yb_1,\dots,\yb_k)\bigcap\Delta_{2^m}$.

\medskip

 Without loss of generality, we assume ${\sf
  SubSpace}(\yb_2-\yb_1, \ldots, \yb_k-\yb_1)$ has a basis $\{\vb_1,
\ldots, \vb_b\}$ for $b\ge 1$, where the matrix $[\vb_1, \ldots, \vb_b]$ has the form
$$
\left[\begin{matrix} \mathbf{V} \\ \mathbf{I}_b \end{matrix} \right].
$$  To
decide whether the vector $\yb$ is contained in the affine space ${\sf
  Aff}(\yb_1, \ldots, \yb_k)$ or not is equivalent to determining whether the
equation system $x_1 \vb_1 + \cdots + x_b \vb_b = \yb_1 -\yb$ has
soluitons or not.  Due to the special structure of the coefficient
matrix, the variables $x_1, \ldots, x_b$ are uniquely associated with
the last $b$ elements of the vector $\yb_1 -\yb$.  Based on these
observations, we propose the following algorithm.

\begin{algorithm}[H]
  \caption{{\sf BooleanVectorSearch}}
  \label{algo:findunitvector}
  \begin{algorithmic}[1]
    \REQUIRE{A set of generators $\yb_1, \ldots, \yb_k$ of an
    affine subspace;}
    \ENSURE{All of the Boolean vectors contained in this affine
    subspace.}
    \STATE Compute a basis $\{\vb_1, \ldots, \vb_b\}$ of the linear
    subspace generated by $\{\yb_2-\yb_1, \ldots, \yb_k-\yb_1\}$.
    Without loss of generality, we assume the last $b$ rows of the matrix $[\vb_1,
      \ldots, \vb_b]$ is the $b\times b$ identity matrix.
    \label{ln:basis}
    \STATE For $i \in [1, b]$, let $\widetilde{\vb}_i$ and
    $\widetilde{\yb}_1$ be the column vectors consisting of the first
    $2^m-b$ rows of $\vb_i$ and $\yb_1$, respectively.  Let
    $\overline{\yb}_1$ be the column vector consisting of the last $b$
    rows of $\yb_1$.
    \label{ln:partition}
    \STATE Let $\yb = \widetilde{\yb}_1 - [\widetilde{\vb}_1,
    \ldots, \widetilde{\vb}_b] \overline{\yb}_1$. Let $S$ be an empty
    set.
    \label{ln:computey}
    \STATE If $\yb = \delta_{2^m-b}^i$ for some $i\in [1, 2^m-b]$,
    then add $\delta_{2^m}^i$ to $S$.
    \label{ln:ei}
    \STATE For $i = 2^m-b+1, \ldots, 2^m$, if
    $-\widetilde{\vb}_{i-2^m+b} = \yb$, then add $\delta_{2^m}^{i}$ to
    $S$.
    \label{ln:vi}
    %%
    %\IF{$\yb = \delta_{2^m-b}^i$ for some $i\in [1, 2^m-b]$}
    %%
     % \STATE Add $\delta_{2^m}^i$ to $S$.
     % \label{ln:ei}
      %%
     % \ENDIF
      %%
     % \FOR{$i$ from $1$ to $d$}
      %%
     % \IF{$-\widetilde{\vb}_i = \yb$}
      %%
     % \STATE Add $\delta_{2^m}^{i+2^m-b}$ to $S$.
     % \label{ln:vi}
      %%
     % \ENDIF
      %%
     % \ENDFOR
      %%
    \RETURN $S$.
  \end{algorithmic}
\end{algorithm}

\begin{lemma}\label{le:correctnessofalgo1}
  Algorithm \ref{algo:findunitvector} is correct and terminates.
\end{lemma}
\begin{proof}
  The Lines
  \ref{ln:basis} and \ref{ln:partition} of Algorithm
  \ref{algo:findunitvector} lead us to solve the following equation
  system with variables $x_1, \ldots, x_b$:
  \begin{equation}\label{eq:equations}
    \left[\begin{matrix}
        \begin{matrix}
          \widetilde{\vb}_1 & \ldots & \widetilde{\vb}_b
        \end{matrix} \\
        \Ib_{b}
      \end{matrix} \right ]
    \left [ \begin{matrix}
        x_1 \\
        \vdots \\
        x_b
      \end{matrix} \right ]
    = \left[ \begin{matrix}
        \widetilde{\yb}_1 \\
        \overline{\yb}_1
      \end{matrix} \right] - \delta_{2^m}^i,
    \text{ for } i: 1 \le i \le 2^m.
  \end{equation}
  We disscuss Equation \ref{eq:equations} in the following two cases:
  \begin{enumerate}
  \item $1 \le i \le 2^m -b$\newline
    In this case, $(x_1, \ldots,x_b)^t = \overline{\yb}_1$. We need to
    determine whether $\delta_{2^m-b}^i = \widetilde{\yb}_1 -
    [\widetilde{\vb}_1, \ldots, \widetilde{\vb}_b] \overline{\yb}_1$
    or not. This leads us to the Line \ref{ln:ei}.
  \item $2^m-b+1 \le i \le 2^m$ \newline
    In this case, for a fixed $i$, $(x_1, \ldots, x_b)^t =
    \overline{\yb}_1 - \delta_{b}^{i-2^m+b}$. We need to determine
    whether $-\widetilde{\vb}_{i-2^m+b} =  \widetilde{\yb}_1 -
    [\widetilde{\vb}_1, \ldots, \widetilde{\vb}_b] \overline{\yb}_1$
    or not. This leads us to the Line \ref{ln:vi}.
  \end{enumerate}
  The correctness of the algorithm follows immediately.
  The algorithm terminates obviously.
\end{proof}

\begin{lemma}\label{le:requiredoperations4algo1}
  Algorithm \ref{algo:findunitvector} requires at most $O(2^mkb)$
  field operations, where $b$ is the dimension of the affine subspace.
\end{lemma}
\begin{proof}
The column reduced echelon form of $[\yb_1, \ldots, \yb_k]$ can be
obtained by Gauss elimination.
The classical Gauss elimination requires $O(2^mkb)$ field operations,
where $b$ is the rank of $[\yb_1, \ldots, \yb_k]$.
Once we have the column reduced echelon form, the only remaining
computations lie in Line \ref{ln:computey}, which requires at most
$O(2^mb)$ field operations.
The lemma follows immediately.
\end{proof}

We write  {\sf BooleanVectorSearch} as our algorithm that outputs all points in the intersection of  ${\sf
  Aff}(\yb_1,\dots,\yb_k)$ and $\Delta_{2^m}$ from $\yb_1,\dots,\yb_k \in \mathbb{R}^{2^m}$ as the input,  i.e.,
$$
{\sf BooleanVectorSearch}(\yb_1,\dots,\yb_k)={\sf
  Aff}(\yb_1,\dots,\yb_k)\mcap \Delta_{2^m}.
$$

\section{Distributed Boolean Equation Solvers}\label{secsolver}
In this section, we present our algorithms that solve the system of Boolean equations distributedly  based on exact or approximate projection consensus algorithms, given satisfiability of the Boolean equations.

\subsection{The Algorithm with Exact Projection Consensus}

We first present a distributed algorithm for solving the system of Boolean equations  (\ref{bes}) utilizing the three algorithms
{\sf BooleanMatricization},   {\sf DistributedLAE}, {\sf BooleanVectorSearch} as subroutines.  Let ${\sf uniform}([0,1]^{2^m})$ be the uniform distribution over the $2^m$-dimensional cube $[0,1]^{2^m}$.  Let $k_\ast$ be a positive integer.

\begin{algorithm}[ht]
  \caption{{\sf DistributedBooleanEquationSolver}}
  \label{algo:dbes}
  \begin{algorithmic}[1]
    \REQUIRE{Over the network $\mathrm{G}$,  node $i$ holds Boolean equation $f_i(x)=\sigma_i$; nodes communicate with only neighbors on $\mathrm{G}$ about their dynamical states.}
    \ENSURE{Each node $i$ computes $\{x\in\{0,1\}^m: f_j(x)=\sigma_j, \ j=1,\dots,n\}$. }
    \STATE Each node $i$ locally  computes $\mathbf{M}_{f_i}={\sf BooleanMatricization}(f_i)\in \mathbb{R}^2\times\mathbb{R}^{2^m}$  for all $i\in \mathrm{V}$;

     \STATE Each node $i$ assigns linear equation  $\mathscr{E}_i^{\rm b}: \mathbf{H}_i \mathbf{y}=\mathbf{z}_i$ and then the projection onto its solution space $\mathcal{E}_i^{\rm b}$ by $\mathbf{H}_i \leftarrow \mathbf{M}_{f_i}$ and $\mathbf{z}_i\leftarrow\Theta_1(\sigma_i)$;
    \STATE For $s=1,\dots, k_\ast $, each node $i$ randomly and independently  selects $\mathbf{x}_i(0)=\bm{\beta}_{i,s}\sim {\sf uniform}([0,1]^{2^m})$, and runs {\sf DistributedLAE} to produce
     $\mathbf{y}_s=\sum_{i=1}^n\mathcal{P}_\ast (\bm{\beta}_{i,s})/n$ at each node $i\in\mathrm{V}$;
    \STATE Each node $i$ locally computes  $S={\sf BooleanVectorSearch}(\yb_1,\dots,\yb_{k_\ast})$;
      \RETURN $\mathcal{S}= \Upsilon_m (S)$ for all nodes.
  \end{algorithmic}
\end{algorithm}

The intuition behind Algorithm \ref{algo:dbes} is as follows. First of all, the algorithm {\sf BooleanMatricization} manages to convert each Boolean equation locally  into the following network linear algebraic equation:
\begin{align}\label{induced}
\mathscr{E}^{\rm b}: \mathbf{M}_{f_i} \mathbf{y}=\Theta_1(\sigma_i),\ \ i=1,\dots,n
\end{align}
which admits a non-empty affine solution subspace $\mathcal{E}^{\rm b}=\mcap_{i=1}^n \mathcal{E}^{\rm b}_i$ when the system of  Boolean equations (\ref{bes})
is solvable.  Then, for any given initial values $(\bm{\beta}_{1,s},\dots,\bm{\beta}_{n,s})$, the algorithm {\sf DistributedLAE} produces an exact algebraic solution $\mathbf{y}_s=\sum_{i=1}^n\mathcal{P}_\ast (\bm{\beta}_{i,s})/n$ to \eqref{induced} in a distributed manner, where with slight abuse of notation $\mathcal{P}_\ast$ is the projection onto $\mathcal{E}^{\rm b}$. However, a single  algebraic solution cannot be used to infer any  solution to the original Boolean equations in the logical space. To overcome this, we randomly select initial values of the nodes for {\sf DistributedLAE} to produce $k_*$ algebraic solutions, in the hope that these random algebraic solutions are useful enough to reconstruct the entire solution affine space of the linear equations.  Finally, by {\sf BooleanVectorSearch} we  search all Boolean vectors in that affine space, from which we eventually uncover the solution set of the Boolean equations by the mapping $\Upsilon_m(\cdot)$.

For the correctness of the Algorithm \ref{algo:dbes}, we present the following theorem.

\begin{theorem}\label{thm1} Suppose the system of  Boolean equations (\ref{bes}) admits at least one exact solutions. Let $k_\ast=2^m+1$.
 Then with probability one,  the set $\mathcal{S}$ that the algorithm {\sf DistributedBooleanEquationSolver} returns is exactly  the set of solutions to (\ref{bes}).
\end{theorem}

Theorem \ref{thm1} suggests that, in order to return a correct output,  Algorithm 2 involves finding the solutions to a system of linear equations with $2^m$ unknowns and $2n$ equations. Multiplying the $\mathbf{H}^\top=[\mathbf{H}_1\dots\mathbf{H}_n]^\top$ from the left and right sides of that equation (which takes $O(4^m n)$ operations), we obtain a system of linear equation of $2^m$ unknowns with $2^m$ equations, whose computational complexity for centralized algorithms is $O\big((2^m)^3\big)$. From Lemma~\ref{le:requiredoperations4algo1}, the Boolean vector search step
takes $O\big((2^m)^3\big)$ steps as well in the worst case ($k=b=O(2^m)$). As a result, regardless of the distributedness of the algorithms for the linear equations, the scheme of taking the computations into the algebraic space has a conservative  computational complexity of the order $O(2^{3m}+2^{2m}n)$ in terms of field operations. Also, at each round of the recursion of the distributed linear algebraic equation solving, a $2^m$ dimensional vector is communicated along each communication link, leading to a significant communication complexity as well. On the other hand, the    linear algebraic equations have sparse structures; the same linear equation is solved for a number of times with different initial values. These factors may be explored to develop accelerated or parallel versions of Algorithm~\ref{algo:dbes}, potentially  reducing the cost from both the computation and communication sides.

\noindent{\bf Example 1}. We present an example illustrating the computation process of  Algorithm \ref{algo:dbes}. Let ``$\wedge$'', ``$\vee$'', and ``$\neg$'' be the logical ``AND'', ``OR", ``NOT" operations. Let ``$\leftrightarrow$" be the logical equivalence operation with $x\leftrightarrow y=(\neg x\vee y)\wedge(\neg y\vee x)$. Consider the following Boolean equations
	\begin{align}\label{ex:eq2}
	\begin{cases}
	f_1=x_1 \vee x_2 \vee \neg x_3 =1,\\
	f_2=x_1\wedge (x_1\leftrightarrow x_2) = 0,\\
	f_3=x_2 \wedge x_3 = 0.
	\end{cases}
	\end{align}
Let there be a network with three nodes in $\mathrm{V}=\{1,2,3\}$ with node $i$ holding the $i$th equation. 	A detailed breakdown of  Algorithm \ref{algo:dbes} is as follows.
	\begin{description}
		\item[S1.] Each node $i$ computes the ${\bf M}_{f_i}$ based on her local Boolean equation $f_i$ as
		\begin{align*}
		&{\bf M}_{f_1} = \begin{bmatrix}
		0&1&0&0&0&0&0&0\\
		1&0&1&1&1&1&1&1
		\end{bmatrix},\\
		&{\bf M}_{f_2} = \begin{bmatrix}
		1&1&1&1&1&1&0&0\\
		0&0&0&0&0&0&1&1
		\end{bmatrix},\\
		&{\bf M}_{f_3} = \begin{bmatrix}
		1&1&1&0&1&1&1&0\\
		0&0&0&1&0&0&0&1
		\end{bmatrix}.
		\end{align*}
		\item[S2.] Each node $i$ locally assigns $\mathscr{E}_i^b$: ${\bf H}_i{\bf y}={\bf z}_i$ by ${\bf H}_i\leftarrow {\bf M}_{f_i}$ and ${\bf z}_i\leftarrow\Theta_1(\sigma_i)$, where
		\begin{align*}
		{\bf z}_1=\Theta_1(1)=\begin{bmatrix}0\\1\end{bmatrix},
		{\bf z}_2=\Theta_1(0)=\begin{bmatrix}1\\0\end{bmatrix},
		{\bf z}_3=\Theta_1(0)=\begin{bmatrix}1\\0\end{bmatrix}.
		\end{align*}
		\item[S3.] Then for $s=1,\dots,2^3+1$, each node $i$ randomly and independently selects ${x}_i(0)=\beta_{i,s}$ from $\mathsf{Uniform}([0,1]^8)$, and runs $\mathsf{DistributedLAE}$ to produce ${\bf y}_s$. A sample obtained by such randomization for the ${\bf y}_s$ is
		$$
		\begin{aligned}
		 & ~[{\bf y}_1~\dots~{\bf y}_9]\\
		=&\begin{bmatrix}
		 0.3837& 0.0299&-0.0509& 0.4616& 0.2897& 0.1139& 0.1043& 0.3578& 0.0277\\
		0.0000& 0.0000& 0.0000& 0.0000& 0.0000& 0.0000& 0.0000& 0.0000& 0.0000\\
		0.1019& 0.4064& 0.2581& 0.1935& 0.3299& 0.2565& 0.4819& 0.3105& 0.1353\\
		0.0640& 0.0604& 0.1110& 0.1157& 0.0974& 0.0806& 0.1506& 0.0511& 0.0300\\
		0.0944& 0.1918& 0.3854& 0.2492&-0.1419& 0.1938& 0.0559& 0.2378& 0.3244\\
		0.3561& 0.3116& 0.2964&-0.0201& 0.4249& 0.3551& 0.2073& 0.0429& 0.4827\\
		0.0640& 0.0604& 0.1110& 0.1157& 0.0974& 0.0806& 0.1506& 0.0511& 0.0300\\
		-0.0640&-0.0604&-0.1110&-0.1157&-0.0974&-0.0806&-0.1506&-0.0511&-0.0300
		\end{bmatrix}.
		\end{aligned}
		$$
		\item[S4.] Each node then locally runs the algorithm $\mathsf{BooleanVectorSearch}$, which gives $S=\{\delta_8^1,\delta_8^3,\delta_8^5,\delta_8^6\}$.
		\item[S5.] Each node finally computes and returns $\mathcal{S}=\Upsilon_m(S)={\{[000],[010],[100],[101]\}}$.
	\end{description}
We can easily verify that $\mathcal{S}={\{[000],[010],[100],[101]\}}$ is indeed the solution set of the equation (\ref{ex:eq2}), which provides a validation to the correctness of the Algorithm \ref{algo:dbes}.

\subsection{Extended Algorithm with Approximate Projection Consensus}

In Algorithm \ref{algo:dbes}, the step for distributed linear equation solving assumes accurate computation of the projections, i.e., for $\mathbf{x}_i(0)=\bm{\beta}_{i,s}\sim {\sf uniform}([0,1]^{2^m})$, the algorithm {\sf DistributedLAE} produces
$\mathbf{y}_s=\sum_{i=1}^n\mathcal{P}_\ast (\bm{\beta}_{i,s})/n$ at each node $i\in\mathrm{V}$. However, based on Lemma \ref{prop2}, the convergence of {\sf DistributedLAE} is only asymptotic, and therefore on the face value, the accurate projection $\mathbf{y}_s=\sum_{i=1}^n\mathcal{P}_\ast (\bm{\beta}_{i,s})/n$ can only be obtained by an infinite number of steps. Suppose the {\sf DistributedLAE} only runs at each node for $T$ steps. Then for $\mathbf{x}_i(0)=\bm{\beta}_{i,s}\sim {\sf uniform}([0,1]^{2^m})$, the output at each node $i$ becomes
  \begin{align}\label{eqapp}
    \widehat{\mathbf{y}_{i,s}}=\sum_{k=1}^n\mathcal{P}_\ast (\bm{\beta}_{k,s})/n +r_{i,s}(T)=\mathbf{y}_s +r_{i,s}(T)
     \end{align}
where $r_{i,s}(T)$ is a computation residual  incurred at node $i$. Now it is of interest to develop an extended algorithm for solving the Boolean equation with only a finite number of iterations for the projection consensus step.

\begin{remark}
It is also possible to construct the exact projection $\mathbf{y}_s=\sum_{i=1}^n\mathcal{P}_\ast (\bm{\beta}_{i,s})/n$ from a series of node states along the {\sf DistributedLAE} algorithm, e.g., \cite{yangtao}, utilizing the idea of finite-time consensus \cite{sundaram2007,yuan2013}. This method would however rely on additional observability conditions for the network structure $\mathrm{G}$, and the structure $\mathrm{G}$ should also be known to all nodes.
\end{remark}

Based on Lemma \ref{prop2} and noting the exponential rate of convergence for algorithm \eqref{eq2}, there holds
\begin{align}\label{eqrr}
\|r_{i,s}(T) \|\leq C_0 e^{-\gamma_0 T}
\end{align}
for some constants $C_0>0$ and $\gamma_0>0$, where $\|\cdot\|$ represents the $\ell_2$ norm. Now, the new challenge lies in how we should construct the affine space $\mathcal{E}^{\rm b}$ of the solutions to the linear equation \eqref{induced} from the $ \widehat{\mathbf{y}_{i,s}},s=1,\dots,2^{m}+1$. There are two important properties that  $\mathcal{E}^{\rm b}$ should satisfy:
\begin{itemize}
\item[(i)] The distance  between $ \widehat{\mathbf{y}_{i,s}}$ and $\mathcal{E}^{\rm b}$ is upper bounded by $C_0 e^{-\gamma_0 T}$ which decays exponentially as $T$ increases;
\item[(ii)] $\mathcal{E}^{\rm b}\mcap \Delta_{2^m}\neq \emptyset$.
\end{itemize}
If only these two properties are utilized, we end up with a trivial solution of $\mathcal{E}^{\rm b}$ from the
 $ \widehat{\mathbf{y}_{i,s}},s=1,\dots,2^{m}+1$ as the $\mathbb{R}^{2^m}$. Therefore, in order to compute the actual solution space $\mathcal{E}^{\rm b}$, we propose to find an affine subspace with the minimal rank  among all the affine spaces satisfying the two properties.  Let $\mathfrak{A}$ denote the set of all affine spaces of $\mathbb{R}^{2^m}$.  Fixing $\epsilon>0$, we  define the following optimization problem:
\begin{equation}\label{opt}
\begin{aligned}
\min_{\mathcal{A}\in \mathfrak{A}} \quad & {\rm dim}(\mathcal{A})\\
\textrm{s.t.} \quad & \sum_{s=1}^{2^{m}+1}{\rm dist}(\widehat{\mathbf{y}_{i,s}},\mathcal{A})\leq \epsilon. \\
\end{aligned}
\end{equation}

The optimization problem (\ref{opt}) is obviously   feasible for all $\epsilon>0$. We present the following lemma on whether solving  (\ref{opt}) can potentially give us the true  $\mathcal{E}^{\rm b}$.
\begin{lemma}\label{lem5}
 Suppose the algorithm {\sf DistributedLAE}   runs at each node for $T$ steps, and produces $\widehat{\mathbf{y}_{i,s}}$ at each node $i$ for $\mathbf{x}_i(0)=\bm{\beta}_{i,s}\sim {\sf uniform}([0,1]^{2^m})$, $s=1,\dots,2^m+1$. Let $\epsilon:=\epsilon_T=C_\ast e^{-\gamma_\ast T}(2^m+1)$ for some $C_\ast\geq C_0$ and some $0<\gamma_\ast \leq \gamma_0$. Then
$$
\lim_{T\to \infty} \mathbb{P}\big(\mathcal{E}^{\rm b}\ \mbox{is a solution to (\ref{opt})} \big)=1
$$
under the probability measure introduced by the random initial values.
\end{lemma}

Based on Lemma \ref{lem5}, with $\epsilon:=\epsilon_T=C_\ast e^{-\gamma_\ast T}(2^m+1)$,  the optimization problem (\ref{opt}) achieves its optimum at  a family of affine spaces $\mathfrak{A}_{i,\epsilon}^\star$.
  Lemma \ref{lem5} also suggests that with high probability, $\mathcal{E}^{\rm b}\in \mathfrak{A}_{i,\epsilon}^\star$ for small $\epsilon$ (i.e., large $T$). It remains unclear how we can localize the exact $\mathcal{E}^{\rm b}$ from $\mathfrak{A}_{i,\epsilon}^\star$, for which we propose to solve the following optimization problem:
  \begin{equation}\label{opt2}
\begin{aligned}
\max_{\mathcal{A}\in  \mathfrak{A}_{i,\epsilon}^\star} \quad & {\rm Card}(\mathcal{A}\mcap \Delta_{2^m}).
\end{aligned}
\end{equation}
Here  ${\rm Card}(\cdot)$ represents the cardinality of a finite set.
\begin{lemma}\label{lem6} Suppose the algorithm {\sf DistributedLAE}   runs at each node for $T$ steps, and produces $\widehat{\mathbf{y}_{i,s}}$ at each node $i$ for $\mathbf{x}_i(0)=\bm{\beta}_{i,s}\sim {\sf uniform}([0,1]^{2^m})$, $s=1,\dots,2^m+1$. Let $\epsilon:=\epsilon_T=C_\ast e^{-\gamma_\ast T}(2^m+1)$ for some $C_\ast\geq C_0$ and some $0<\gamma_\ast \leq \gamma_0$.
There holds
$$
\lim_{T\to \infty} \mathbb{P}\big(\mathcal{E}^{\rm b}\in   \mathfrak{A}_{i,\epsilon}^\star\big)=1
$$
under the probability measure introduced by the random initial values.
\end{lemma}

Let $C_*\ge C_0,~\gamma_*\le\gamma_0$ be fixed. Let $T$ be given.
We present the following algorithm which relies only on a finite number of steps for the {\sf DistributedLAE} subroutine.
\begin{algorithm}[H]
  \caption{{\sf ExtendedDistributedBooleanEquationSolver}}
  \label{algo:ag2}
  \begin{algorithmic}[1]
    \REQUIRE{Over the network $\mathrm{G}$,  node $i$ holds Boolean equation $f_i(x)=\sigma_i$; nodes communicate with only neighbors on $\mathrm{G}$ about their dynamical states.}
    \ENSURE{Each node $i$ computes $\{x\in\{0,1\}^m: f_j(x)=\sigma_j, \ j=1,\dots,n\}$ with high probability. }
    \STATE Each node $i$ locally  computes $\mathbf{M}_{f_i}={\sf BooleanMatricization}(f_i)\in \mathbb{R}^2\times\mathbb{R}^{2^m}$  for all $i\in \mathrm{V}$;

     \STATE Each node $i$ assigns  $\mathscr{E}_i^{\rm b}: \mathbf{H}_i \mathbf{y}=\mathbf{z}_i$ by $\mathbf{H}_i \leftarrow \mathbf{M}_{f_i}$ and $\mathbf{z}_i\leftarrow\Theta_1(\sigma_i)$;
    \STATE For $s=1,\dots, 2^m+1$, each node $i$ randomly and independently  selects $\mathbf{x}_i(0)=\bm{\beta}_{i,s}\sim {\sf uniform}([0,1]^{2^m})$, and runs {\sf DistributedLAE} for $T$ steps to produce
     $\widehat{\mathbf{y}_{i,s}}=\sum_{k=1}^n\mathcal{P}_\ast (\bm{\beta}_{k,s})/n +r_{i,s}(T)$ at each node $i\in\mathrm{V}$;
    \STATE Each node $i$ assigns $\epsilon=C_\ast e^{-\gamma_\ast T}(2^m+1)$ to produce $\mathfrak{A}_{i,\epsilon}^\star$ by solving (\ref{opt}) locally;
    \STATE Each node $i$ solves (\ref{opt2})  to generate $S_i={\rm arg}\max_{\mathcal{A}\in  \mathfrak{A}_{i,\epsilon}^\star}   {\rm Card}(\mathcal{A}\mcap \Delta_{2^m}) $ locally;

      \RETURN $\mathcal{S}_i= \Upsilon_m (S_i)$ at node $i$ for all $i\in\mathrm{V}$.
  \end{algorithmic}
\end{algorithm}

\begin{theorem}\label{thm2} Suppose the system of  Boolean equations (\ref{bes}) admits at least one exact solutions.
The Algorithm \ref{algo:ag2} returns    the set of solutions to (\ref{bes}) at each node $i$ with high probability for large $T$. To be precise, there holds for all $i=1,\dots,n$ that
$$
\lim_{T\to \infty} \mathbb{P}\big(\mathcal{S}_i\ \mbox{is the set of solutions to (\ref{bes})} \big)=1.
$$
  \end{theorem}

\begin{remark}
Algorithm \ref{algo:ag2} relies on knowledge at each node on an upper bound of $C_0$, and a lower bound of $\gamma_0$. Since the number of Boolean equations in the form of (\ref{bes}) is essentially finite\footnote{There can be an infinite number of formulas representing a single Boolean mapping $g$ from $\{0,1\}^m$ to $\{0,1\}$. However, the matrix representation  $\mathbf{M}_g$ in the sense of Lemma \ref{prop1} is unique which does not depend on a particular formula of $g$.}, such knowledge can be obtained from the network structure $\mathrm{G}$. Note that Algorithm \ref{algo:dbes} on the other hand does not rely on any information of $\mathrm{G}$.
\end{remark}
It is straightforward to see that Theorem \ref{thm2}  is a direct result of Lemma \ref{lem5} and Lemma \ref{lem6}. The  details of the proofs of Lemma \ref{lem5} and Lemma \ref{lem6} have been put in the appendices. Moreover,  it would be of interest to understand the convergence time of Algorithm \ref{algo:ag2}, which  can be measured as the 
$$
T_\epsilon:= \inf \Big\{T: \  \mathbb{P}\big(\mathcal{S}_i\ \mbox{is the set of solutions to (\ref{bes})} \big) \geq 1 - \epsilon \Big\}. 
$$
An estimate of $T_\epsilon$ as a function of $\epsilon$ relies on estimates of $\mathbb{P}\big(\mathcal{E}^{\rm b}\ \mbox{is the unique solution to (\ref{opt2})} \big)$  as a function of $T$. This is however not possible under the current analysis, as Lemma 5 is established by a contradiction argument.

\subsection{Prior Knowledge of Solution Set}
From Theorem \ref{thm1}, to find all solutions to the  Boolean equations (\ref{bes}), the projection consensus step for solving the induced linear algebraic equation needs to run $2^m+1$ rounds with randomly selected initial values. It turns out, if we know  certain structure of the Boolean equations (\ref{bes}), we can reduce the number of rounds of solving linear equations.

Let $\chi_0$ be the cardinality of the image of the Boolean mapping $[f_1\ \dots\ f_n]^\top$, i.e.,
\begin{align}
\chi_0={\rm Card}\big\{(f_1(x),\ \dots, f_n(x)): x\in \{0,1\}^m\big\}.
\end{align}
We present the following result.

\begin{theorem}\label{thm3} Suppose the system of  Boolean equations (\ref{bes}) admits at least one exact solutions.   Let $k_\ast=2^m-\chi_0+1$.
 Then with probability one,  the set $\mathcal{S}$ that the algorithm {\sf DistributedBooleanEquationSolver} returns is the solution set  to   (\ref{bes}).
\end{theorem}

The detailed proof of  Theorem    \ref{thm3} is in the appendix. As an illustration of Theorem \ref{thm3}, we present the following example.

\noindent {\bf Example 2}.
Let ``$\rightarrow$" denote the logical implication operation with $x\rightarrow y=\neg x\vee y$.
Consider the following Boolean equations
	\begin{align}\label{ex:eq}
	\begin{cases}
		f_1(x_1,x_2,x_3)=(x_1 \vee x_2) \wedge \neg x_3 =1,\\
		f_2(x_1,x_2,x_3)=(x_1\rightarrow x_2) \vee x_3 = 0,\\
		f_3(x_1,x_2,x_3)=x_1 \wedge x_3 = 0
	\end{cases}
	\end{align}
	which has a unique solution $[100]\in\{0,1\}^3$. We can verify that $\chi_0=4$ for the given $f_i,i=1,\dots,3$.  Along the Algorithm \ref{algo:dbes}, each node $i$ computes the ${\bf M}_{f_i}$ from $f_i$ as
	\begin{equation}\label{m}
\begin{aligned}
	& {\bf M}_{f_1} = \begin{bmatrix}
		1&1&0&1&0&1&0&1\\
		0&0&1&0&1&0&1&0
	  \end{bmatrix},\\
	& {\bf M}_{f_2} = \begin{bmatrix}
	  0&0&0&0&1&0&0&0\\
	  1&1&1&1&0&1&1&1
	  \end{bmatrix},\\
	&  {\bf M}_{f_3} = \begin{bmatrix}
	  1&1&1&1&1&0&1&0\\
	  0&0&0&0&0&1&0&1
	  \end{bmatrix}.\\
	\end{aligned}
	\end{equation}
Then  node $i$ locally assigns $\mathscr{E}^{\rm b}_i: \mathbf{H}_i \mathbf{y}=\mathbf{z}_i$ by $\mathbf{H}_i \leftarrow \mathbf{M}_{f_i}$ and $\mathbf{z}_i\leftarrow\Theta_1(\sigma_i)$, where
	\begin{align*}
	{\bf z}_1=\Theta_1(1)=\begin{bmatrix}0\\1\end{bmatrix},
	{\bf z}_2=\Theta_1(0)=\begin{bmatrix}1\\0\end{bmatrix},
	{\bf z}_3=\Theta_1(0)=\begin{bmatrix}1\\0\end{bmatrix}.
	\end{align*}

Next, for $s=1,\dots,2^3-\chi_0+1=5$, each node $i$ randomly and independently selects ${x}_i(0)=\beta_{i,s}$ from $\mathsf{Uniform}([0,1]^8)$, and runs $\mathsf{DistributedLAE}$ to produce ${\bf y}_s$ as
	\begin{align*}
	[\mathbf{y}_1\ \dots \ \mathbf{y}_5] %\nonumber\\
	=\begin{bmatrix}
-0.1558& 0.0871&-0.1208&-0.0962&-0.1209\\
 0.1417&-0.1609& 0.1522& 0.1201&-0.0082\\
-0.0003& 0.0835& 0.0835&-0.1127& 0.1244\\
 0.0141& 0.0738&-0.0314&-0.0239& 0.1292\\
 1.0000& 1.0000& 1.0000& 1.0000& 1.0000\\
-0.0813& 0.0717&-0.1067& 0.1856&-0.0769\\
 0.0003&-0.0835&-0.0835& 0.1127&-0.1244\\
 0.0813&-0.0717& 0.1067&-0.1856& 0.0769
 \end{bmatrix}.
\end{align*}
Each node then locally runs the algorithm $\mathsf{BooleanVectorSearch}$, which gives
	$S=\{\delta_8^5\}$. Eventually nodes return  $\mathcal{S}= \Upsilon_m (S)=\{[100]\}$. We can randomly select other  ${x}_i(0)=\beta_{i,s}$ for $s=1,\dots,5$, and $\mathcal{S}=\{[100]\}$ is always correctly computed. Therefore, we have provided a verification of Theorem \ref{thm3}.

\section{Distributed Satisfiability  Verification}\label{secsat}

The correctness of the Algorithm \ref{algo:dbes} and Algorithm \ref{algo:ag2} relies on the crucial fact that   the system of Boolean equations (\ref{bes}) admits at least one exact solutions. This, however, has no guarantee in the first place. Even with all the $f_i$, verification of this satisfiability is a classical {\sf SAT} problem. For the network that runs  Algorithm \ref{algo:dbes} and Algorithm \ref{algo:ag2}, we also need to develop a method that can verify the satisfiability of (\ref{bes}) in a distributed manner.
\subsection{A Distributed {\sf SAT}  Algorithm}
We present the following Algorithm 4 for distributedly verifying the satisfiability of (\ref{bes}),  along with  a correctness proof of the algorithm.
\begin{algorithm}[ht]
  \caption{{\sf DistributedBooleanSatisfiability}}
  \label{algo:solbability}
  \begin{algorithmic}[1]
    \REQUIRE{Over the network $\mathrm{G}$,  node $i$ holds Boolean equation $f_i(x)=\sigma_i$; nodes communicate with only neighbors on $\mathrm{G}$ about their dynamical states.}
    \ENSURE{Each node $i$ verifies definitively satisfiability of $\{x\in\{0,1\}^m: f_j(x)=\sigma_j, \ j=1,\dots,n\}$ for all $i$. }
    \STATE Each node $i$ locally  computes $\mathbf{M}_{f_i}={\sf BooleanMatricization}(f_i)\in \mathbb{R}^2\times\mathbb{R}^{2^m}$  for all $i\in \mathrm{V}$;

     \STATE Each node $i$ assigns  $\mathscr{E}_i^{\rm b}: \mathbf{H}_i \mathbf{y}=\mathbf{z}_i$ by $\mathbf{H}_i \leftarrow \mathbf{M}_{f_i}$ and $\mathbf{z}_i\leftarrow\Theta_1(\sigma_i)$;
    \STATE Each node $i$ randomly and independently  selects $\mathbf{x}_i(0)\sim {\sf uniform}([0,1]^{2^m})$, and runs the recursion (\ref{eq2}) to produce an output $\tilde{\mathbf{y}}_i$;
       \STATE The network runs the average consensus algorithm (\ref{consensus}) with $\mathbf{x}_i(0)=\tilde{\mathbf{y}}_i$ to produce an output $\tilde{\mathbf{y}}_{\rm ave}=\sum_{j=1}^n \tilde{\mathbf{y}}_j/n$ at each node $i$;

          \STATE   {\bf return} {\sf unsatisfiable} at node $i$ if $\tilde{\mathbf{y}}_{\rm ave}=\tilde{\mathbf{y}}_i$; and go to Step \ref{s5} otherwise;

          \STATE
          The network runs {\sf DistributedLAE} to produce
     $\mathbf{y}_s=\sum_{i=1}^n\mathcal{P}_\ast (\bm{\beta}_{i,s})/n$ at each node $i\in\mathrm{V}$ with $\mathbf{x}_i(0)=\bm{\beta}_{i,s}\sim {\sf uniform}([0,1]^{2^m})$, $s=1,\dots,2^m+1$, and then each node $i$ locally computes  $\mathcal{S}= \Upsilon_m\big({\sf BooleanVectorSearch}(\yb_1,\dots,\yb_{2^m+1})\big)$;   \label{s5}
          \RETURN {\sf unsatisfiable} if $\mathcal{S}=\emptyset$, and {\sf satisfiable} otherwise at each node $i$.
  \end{algorithmic}
\end{algorithm}

\begin{theorem} \label{thm4}
With probability one, the Algorithm \ref{algo:solbability} correctly  returns     the satisfiability    of (\ref{bes}) at all nodes.
\end{theorem}

There are a few points behind  the Algorithm \ref{algo:solbability} that are worth emphasizing. First of all, satisfiability of the Boolean equations (\ref{bes}) is not equivalent to the satisfiability of the induced algebraic equation $\mathscr{E}^{\rm b}$ in (\ref{induced}). In fact, $\mathscr{E}^{\rm b}$ may be satisfiable in $\mathbb{R}^{2^m}$, but not in $\Delta_{2^m}$ which corresponds to the solutions to the Boolean equations (\ref{bes}). Secondly,   Algorithm \ref{algo:dbes} cannot be directly applied when the Boolean equations (\ref{bes}) is unsatisfiable since the Step $3$ of
 Algorithm \ref{algo:dbes} depends  crucially on the fact that  $\mathscr{E}^{\rm b}$ is solvable.  Therefore, in Algorithm \ref{algo:solbability}, we embed a component where nodes can first distributedly verify the satisfiability of  the induced algebraic equation $\mathscr{E}^{\rm b}$ as a preliminary evaluation of satisfiability for (\ref{bes}), and then the subroutines of Algorithm \ref{algo:dbes} can be utilized to produce  a further and final  decision on the satisfiability for (\ref{bes}).

 In order to establish the desired theorem, we would require the following  technical lemma.
\begin{lemma}\label{lem7}
   Suppose the system of linear equations (\ref{LAE}) is not satisfiable. Then along the recursion (\ref{eq2}), there hold
   \begin{itemize}
   \item[(i)] Each $\mathbf{x}_i(t)$    converges to a finite value $\mathbf{y}_i^\ast$ as $t\to\infty$.
   \item[(ii)] There exist at least two nodes $j,k\in\mathrm{V}$ such that $\mathbf{y}_j^\ast\neq \mathbf{y}_k^\ast$.
   \end{itemize}
   \end{lemma}

The proof of Lemma  \ref{lem7} is put in the appendix, followed by the proof of Theorem \ref{thm4}.

 \subsection{Numerical Examples}
\noindent {\bf Example 3}. Consider the following Boolean equations
\begin{align}\label{ex:eq3}
	\begin{cases}
	f_1=x_1 \wedge x_2 \wedge x_3 =1,\\
	f_2=\neg x_1\vee (x_2\leftrightarrow x_3) = 1,\\
	f_3=x_1\wedge ( x_2\vee x_3) = 0.
	\end{cases}
	\end{align}
	which is unsatisfiable. The $f_i,i=1,\dots,3$ lead to    ${\bf M}_{f_i}$ as:
	\begin{align*}
	&{\bf M}_{f_1} = \begin{bmatrix}
	1&1&1&1&1&1&1&0\\
	0&0&0&0&0&0&0&1
	\end{bmatrix},\\
	&{\bf M}_{f_2} = \begin{bmatrix}
	0&0&0&0&0&1&1&0\\
	1&1&1&1&1&0&0&1
	\end{bmatrix},\\
	&{\bf M}_{f_3} = \begin{bmatrix}
	1&1&1&1&1&0&0&0\\
	0&0&0&0&0&1&1&1
	\end{bmatrix};
	\end{align*}
	and the $\mathbf{z}_i$ give
		\begin{align*}
	{\bf z}_1=\Theta_1(1)=\begin{bmatrix}0\\1\end{bmatrix},
	{\bf z}_2=\Theta_1(1)=\begin{bmatrix}0\\1\end{bmatrix},
	{\bf z}_3=\Theta_1(0)=\begin{bmatrix}1\\0\end{bmatrix}.
	\end{align*}
Then  node $i$ locally assigns $\mathscr{E}^{\rm b}_i: \mathbf{H}_i \mathbf{y}=\mathbf{z}_i$ by $\mathbf{H}_i \leftarrow \mathbf{M}_{f_i}$ and $\mathbf{z}_i\leftarrow\Theta_1(\sigma_i)$.

\begin{figure}[H]
\centering
\includegraphics[scale=0.2]{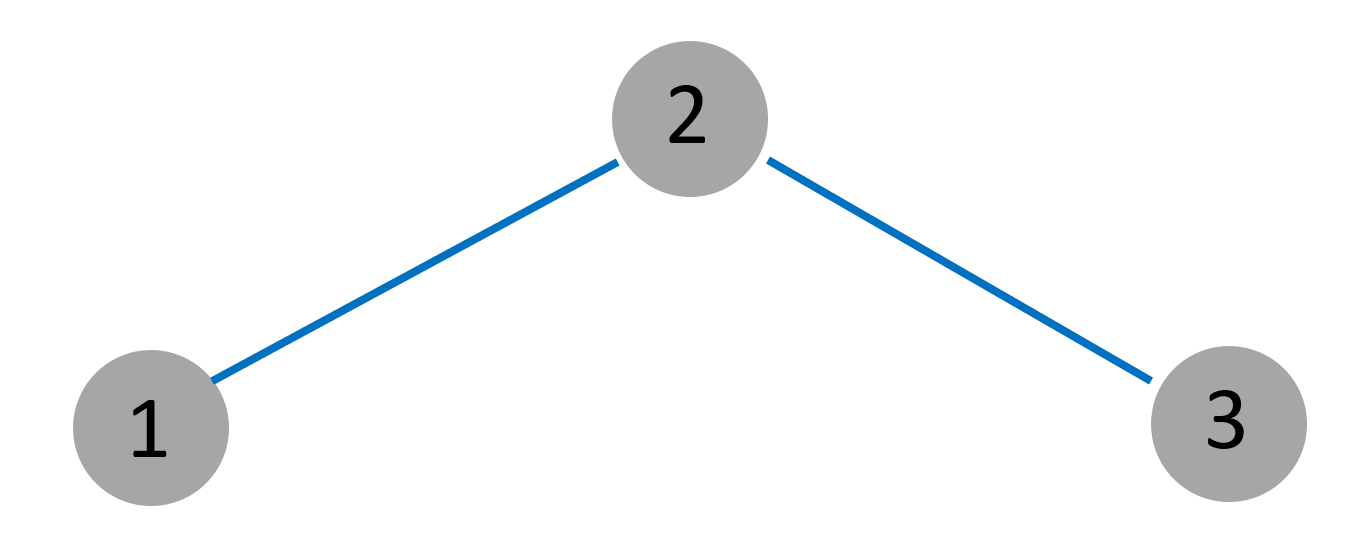}
\caption{The three-node path graph.}\label{fig2}
\end{figure}

Let the network $\mathrm{G}$ be a three-node path graph as shown in Figure \ref{fig2}. Each node $i$ randomly and independently  selects $\mathbf{x}_i(0)\sim {\sf uniform}([0,1]^{2^m})$, and runs the algorithm (\ref{eq2}) with $\epsilon=0.2$, where each $\mathbf{x}_i(t)$ is in $\mathbb{R}^8$. Denote $\mathbf{x}_i^\sharp(t):=(\mathbf{x}_i(t)_{[1]}\ \mathbf{x}_i(t)_{[2]})^\top$, where $\mathbf{x}_i(t)_{[1]}$ and $\mathbf{x}_i(t)_{[2]}$ are, respectively, the first and second entries of $\mathbf{x}_i(t)$. We run the algorithm (\ref{eq2}) for $50$ steps, and plot the trajectories of $\mathbf{x}_i^\sharp(t),t=0,1,\dots,50$ for $i=1,2,3$, respectively, in Figure \ref{figex3}.
	\begin{figure}[H]
		\centering
		\includegraphics[width=10cm]{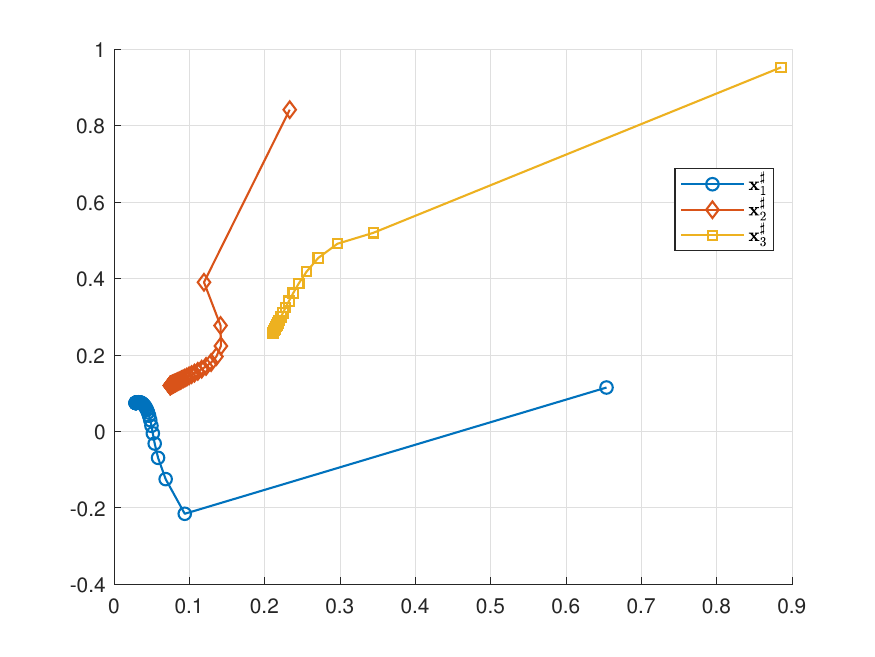}
		\caption{The trajectories of $\mathbf{x}_i^\sharp(t),t=0,1,\dots,50$ for $i=1,2,3$ with a randomly selected initial condition.} \label{figex3}
	\end{figure}
From the trajectories in Figure \ref{figex3}, it can be seen that $\mathbf{x}_i^\sharp(t)$ (and thus $\mathbf{x}_i(t)$) converge to different limits. As a result, Algorithm \ref{algo:solbability} returns {\sf unsatisfiable} in Step $5$.

\noindent {\bf Example 4}. Let us consider the following Boolean equations
	\begin{align} \label{ex4}
	\begin{cases}
		f_1(x_1,x_2,x_3)=(x_1 \vee x_2) \wedge \neg x_3 =0,\\
		f_2(x_1,x_2,x_3)=(x_1\rightarrow x_2) \vee x_3 = 0,\\
		f_3(x_1,x_2,x_3)=x_1 \wedge x_3 = 1
	\end{cases}
	\end{align}
	which is unsatisfiable. The $f_i$'s are the same as those used in Example 2, and therefore lead to the same $\mathbf{M}_{f_i}$'s in (\ref{m}).
Then  node $i$ locally assigns $\mathscr{E}^{\rm b}_i: \mathbf{H}_i \mathbf{y}=\mathbf{z}_i$ by $\mathbf{H}_i \leftarrow \mathbf{M}_{f_i}$ and $\mathbf{z}_i\leftarrow\Theta_1(\sigma_i)$, where
	\begin{align*}
	{\bf z}_1=\Theta_1(0)=\begin{bmatrix}1\\0\end{bmatrix},
	{\bf z}_2=\Theta_1(0)=\begin{bmatrix}1\\0\end{bmatrix},
	{\bf z}_3=\Theta_1(1)=\begin{bmatrix}0\\1\end{bmatrix}.
	\end{align*}
Although (\ref{ex4}) is unsatisfiable, we can verify that the linear equation $\mathscr{E}^{\rm b}$ from the three $\mathscr{E}^{\rm b}_i$'s is actually satisfiable since
\[
{\rm rank}\begin{bmatrix}
\mathbf{M}_{f_1}\\
\mathbf{M}_{f_2}\\
\mathbf{M}_{f_3}
\end{bmatrix}
={\rm rank}\begin{bmatrix}
\mathbf{M}_{f_1}&  {\bf z}_1\\
\mathbf{M}_{f_2} &  {\bf z}_2\\
\mathbf{M}_{f_3}&  {\bf z}_3
\end{bmatrix}.
\]
Thus, Algorithm \ref{algo:solbability} proceeds to Step $6-7$ after Step $5$, and eventually returns {\sf unsatisfiable} in Step $7$.

\section{Differentially Private Boolean Equation Solver}
In this section, we present an  algorithm for the Boolean equation system that admits differential privacy guarantee under Definition \ref{dp}. The Algorithms 2--4 have three stages:
\begin{itemize}
\item[(i)] (Algebraic transformation) Each Boolean equation is transformed into an algebraic form in a Euclidean space; the Boolean equation system in turn induces  a linear algebraic equation over the network.
\item[(ii)] (Distributed linear algebraic equation solving) All nodes exchange information about their dynamical states so that solutions to the network-level linear equation are obtained at individual nodes.
\item[(ii)] (Boolean vector search) Each node carries out exact or approximate Boolean vector search so that solutions to the Boolean equations are established.
\end{itemize}
Clearly, only at the stage (ii), node-to-node communications are necessary while the computations at stages (i) and (iii) are strictly local. Therefore, as long as we can have mechanisms in place that guarantee differential privacy for the distributed computations of linear algebraic equations, the   Algorithms 2--4 can then be made differentially private. For distributed consensus seeking and continuous optimization with differential privacy, there have been established lines of research with proven privacy guarantees under different conditions \cite{d1,d2,d3}.

\subsection{Algorithm with Differential Privacy}

Consider the linear equation (\ref{LAE}) over the network $\mathrm{G}$. Denote $\Omega:=[-1,1]^k$, where $k$ is the dimension of the decision variable in (\ref{LAE}). Since $\Omega$ is a compact convex set, we can introduce $\Pi_{\Omega}(\cdot)$ as the projector onto $\Omega$ in $\mathbb{R}^k$.  Let $\omega_i(t), i=1,\dots,n, t=0,1,2,\dots$ be i.i.d. Laplace noise with zero mean and variance $\sigma$. We first present the following
extended algorithm for solving the equation (\ref{LAE}):
\begin{align}\label{newlae}
\mathbf{x}_i(t+1)=\Pi_{\Omega}\bigg[\mathcal{P}_i\Big(\mathbf{x}_i(t)+\epsilon \sum_{j\in\mathrm{N}_i}\big(\mathbf{x}_j(t)- \mathbf{x}_i(t)\big)\Big)+\omega_i(t)\bigg], \ \ i=1,\dots,n.
\end{align}
Let  the algorithm (\ref{newlae}) run for $T$ steps starting from $\mathbf{x}(0):=(\mathbf{x}_1(0)\ \dots\ \mathbf{x}_n(0))^\top$, denote the output at each node $i$ as $\mathbf{x}_i(T)=\mathcal{D}_{i,T}(\mathbf{x}(0))$. We present the following algorithm for solving the Boolean equation (\ref{bes}).

\begin{algorithm}[H]
  \caption{{\sf DP-DistributedBooleanEquationSolver}}
  \label{algo:dp}
  \begin{algorithmic}[1]
    \STATE Each node $i$ locally  computes $\mathbf{M}_{f_i}={\sf BooleanMatricization}(f_i)\in \mathbb{R}^2\times\mathbb{R}^{2^m}$  for all $i\in \mathrm{V}$;

     \STATE Each node $i$ assigns linear equation  $\mathscr{E}_i^{\rm b}: \mathbf{H}_i \mathbf{y}=\mathbf{z}_i$ and then the projection onto its solution space $\mathcal{E}_i^{\rm b}$ by $\mathbf{H}_i \leftarrow \mathbf{M}_{f_i}$ and $\mathbf{z}_i\leftarrow\Theta_1(\sigma_i)$;
    \STATE For $s=1,\dots,K $, each node $i$ randomly and independently  selects $\mathbf{x}_i(0)=\bm{\beta}_{i,s}\sim {\sf uniform}([0,1]^{2^m})$, and the network runs (\ref{newlae}) for $T$ steps to produce
     $\mathbf{y}_{i,s}:=\mathcal{D}_{i,T}(\bm{\beta}_{s}) \in \mathbb{R}^{2^m}$ at each node $i$;
    %%
  %  \STATE Each node $i$ locally computes $\mathbf{y}_{i,\mu}:=\sum_{s=(\mu-1)l+1}^{\mu l} \mathbf{y}_{i,s}/l$ for $\mu=1,\dots,K$;
    \STATE Each node $i$ solves
     $\mathbf{O}_i= \argmin_{x\in \mathbb{R}^{2^m},\|x\|=1 } \sum_{s=1}^{K}\Big(\big(\mathbf{y}_{i,s}-\mathbf{y}_{i,1}\big)^\top x\Big)^2$, and selects uniformly at random $o_{1,i},\dots,o_{w,i}\in \mathbf{O}_i$;
    \STATE
    Each node $i$ computes  $l_\ast\in \Delta_{2^m}$ satisfying  $
  l_i^\ast\in \argmin_{l\in \Delta_{2^m}} \sum_{j=1}^{w}\Big( ({l-\mathbf{y}_{i,1}})^\top  o_{j,i}  \Big)^2    $;
     \RETURN $s^\ast_i= \Upsilon_m (  l_i^\ast)$ at each node $i$.
  \end{algorithmic}
\end{algorithm}

The essential idea behind Algorithm \ref{algo:dp} lies in twofolds: the exploration of the Laplace mechanism in the algebraic equation solving part to ensure differential privacy (steps 2 -- 3), and
search for one Boolean vector corresponding to the Boolean equation from points near the affine manifold of the algebraic equation (steps 4 -- 5).   The following result holds.

\begin{theorem}\label{thm-DP}
Suppose the system of  Boolean equations (\ref{bes}) admits at least one exact solution.  Let $\kappa = 6+(2+6\times2^{{(m/2)}})2^{m}$.   Then the Algorithm \ref{algo:dp} is $\epsilon$-differentially private under  Definition \ref{dp} if $\sigma \geq\kappa T K/ \epsilon$.
\end{theorem}

Theorem \ref{thm-DP} is established via a careful analysis of the relationship between the differential privacy of the Boolean equation and the resulting algebraic equation. The privacy preservation does not come for free. In general, the probability that the Algorithm \ref{algo:dp} returns a correct solution to (\ref{bes}), depends on the values $\sigma$ and $T$ and is thus denoted as  $\mathbb{P}_{\sigma,T}$. The privacy level increases, while  $\mathbb{P}_{\sigma,T}$ obviously decreases,   as $\sigma$ increases and $T$ decreases. This is a reflection of the privacy vs. accuracy tradeoff in differentially private computing and optimization schemes.  Letting $K=w=2^m+1$, similar to Theorem \ref{thm2}, there should hold
$
\lim_{\sigma\to 0, T\geq \infty}\mathbb{P}_{\sigma,T} =1.
$ One can have both privacy and accuracy guarantee if the steps for solving the induced linear algebraic equations can be made both accurate and private, see \cite{PPSC} for such a proposal on the assumption of a public-private communication network setup. From the proof of Theorem 5, it is also clear that if the subroutines for solving the induced algebraic equations can be made differentially private without losing their convergence accuracy, Algorithms 1 -- 4 can all become differentially private.

\subsection{Numerical Examples}

\noindent{\bf Example 5.} We consider the three-node network and the  Boolean equation investigated in Example 3 again. In Algorithm 5, we set  $K=9$, and approximately solve the $\mathbf{O}_i$ in step 4. Under different levels of noise variance $\sigma$, we run the algorithm for $100$ independent rounds, and record the times when the output $s_i^\ast$ is a true solution to the Boolean equation for node $i$, respectively, for $i=1,2,3$. We also record the number of times where all three $s_i^\ast$ are simultaneously  correct among the 100 rounds. The computation accuracy is then encoded in the rate of correctness.

In Figure 4, the performance of the Algorithm 5 with $T=500$ under different level of injected noises is demonstrated by the rates of correctness.  Clearly, when $\sigma$ tends to be small, the probability that the nodes returns a true solution approaches one. The result also shows that at a relatively high level of injected noise, the algorithm       continues to have some degree of correctness, demonstrating robustness of the algorithm.

\begin{figure}[ht]
		\centering
		\includegraphics[width=10cm]{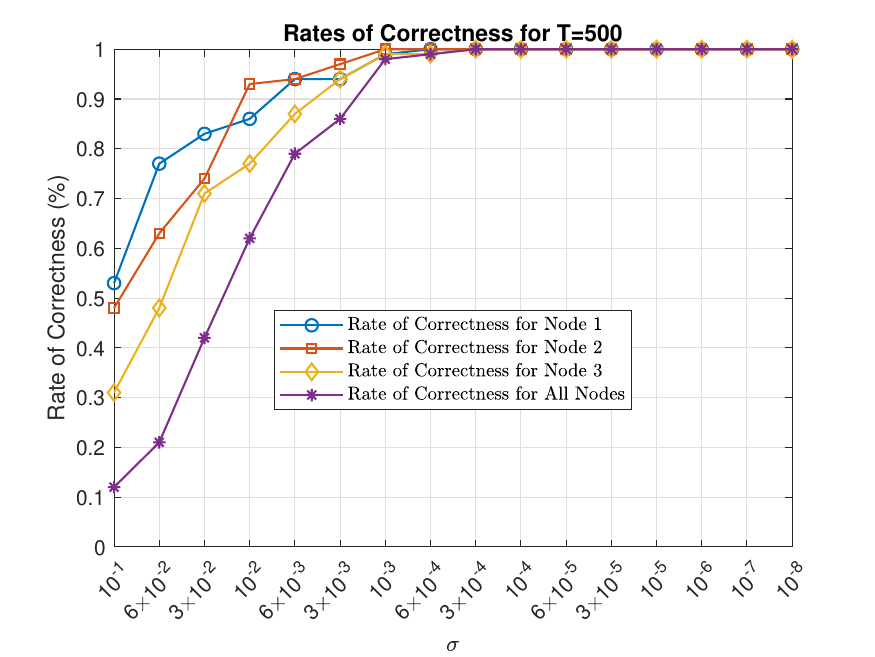}
		\caption{The accuracy level of the Algorithm 5 under different noise levels, measured by the rate of the algorithm returning correct answers among the 100 trials. } \label{figex3}
	\end{figure}

In Figure 5, for $\sigma=10^{-2}$ and $\sigma=10^{-3}$ under different number of steps $T$, the rates of correctness are also shown. It appears that the accuracy of the algorithm does not necessarily grow as $T$ increases with higher level of noises. 
\begin{figure}[H]
\centering
\begin{minipage}[t]{0.48\linewidth}
\centering
\includegraphics[width=3.2in]{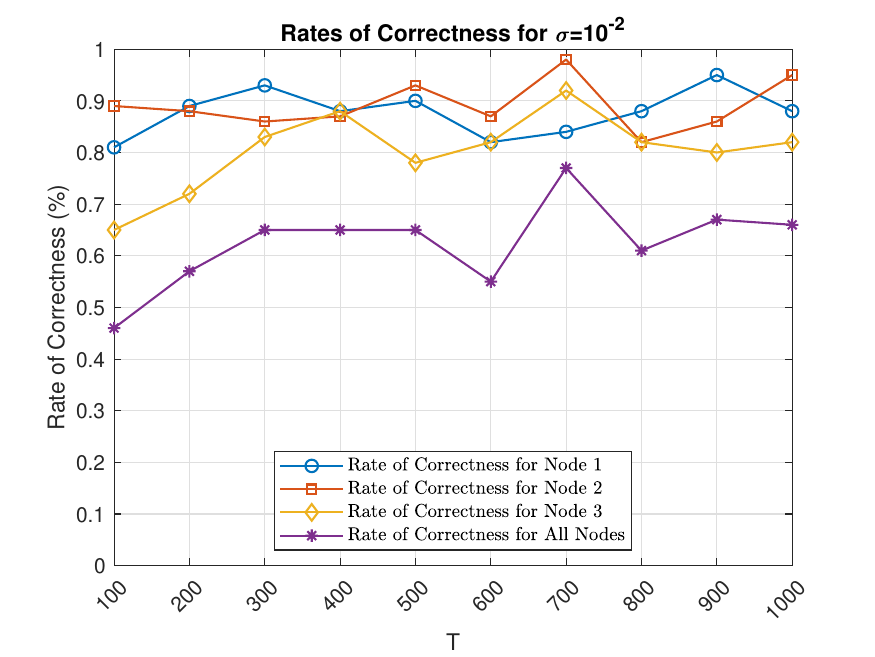}
\end{minipage}
\begin{minipage}[t]{0.48\linewidth}
\centering
\includegraphics[width=3.2in]{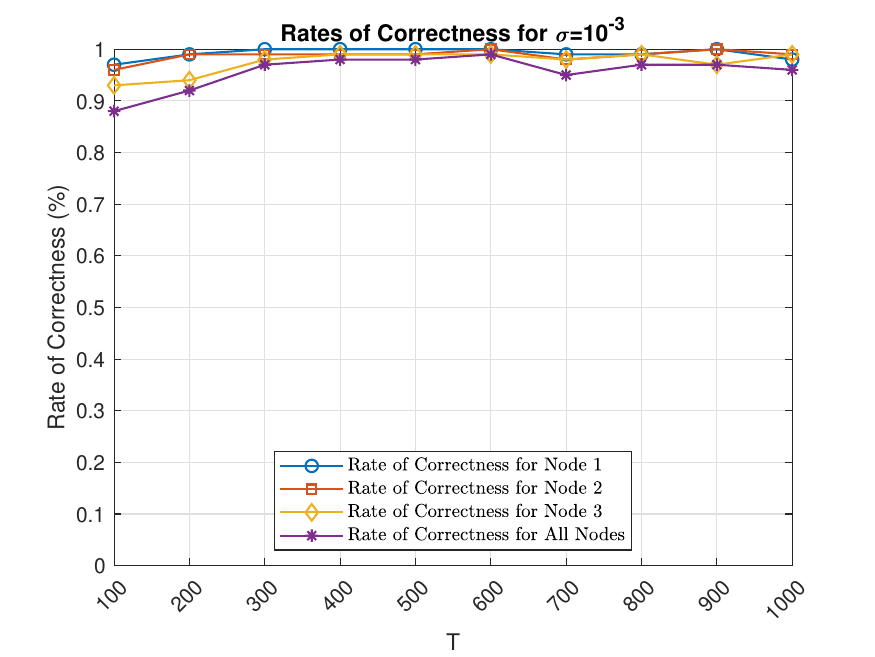}
\end{minipage}
\caption{The rates of returning a correct solution under different $T$ for $\sigma=10^{-2}$ and $\sigma=10^{-3}$, respectively.  }
\end{figure}

\section{Conclusions}\label{seccon}
We have proposed  distributed algorithms that  solve a system of Boolean equations over a network, where each node in the network possesses only  one Boolean equation from the system. The Boolean equation for a particular  node is a private equation known to this only, and the network nodes aim to compute the exact set of solutions to the system of Boolean equations without exchanging their local equations. Based on an algebraic representation of Boolean mappings and existing algorithms for network linear equations, we show that    distributed algorithms, where nodes do not directly reveal their equations or solution sets to neighbors over a graph,  can be constructed that  return the exact solution set of the Boolean equations at each node with high probability.  We also presented  an algorithm for distributed verification of the satisfiability of    Boolean equations. These algorithms hold the promise of becoming proven differentially private if suitable noise-injection mechanisms can be exploited. Therefore, the proposed algorithms put together may serve as  a comprehensive framework  for distributedly treating Boolean equations in terms of verifying satisfiability and then finding the exact solution set whenever possible. 

In the meantime, the work left open several questions. First of all, since the detour into the algebraic space significantly increased communication and computation  complexities, a thorough analysis would be of interest  on the  complexity of the proposed new class of algorithms  in terms of  number of field operations, and number of bits communicated among the nodes. Furthermore, when the system of Boolean equations are not solvable, it would be of interest to introduce proper notions of approximate solutions in the spirit of least squares solutions to systems of algebraic equations. The proposed algebraic embedding of the Boolean equations in Euclidean spaces also introduced natural  geometric metrics, which could shed lights on such extensions.  Finally, whether    the local and (potential) global privacy guarantees can be equivalently obtained for distributed algorithms in the discrete Boolean space, would be an interesting question.

\section*{Appendices}

\subsection*{A. Proof of Theorem \ref{thm1}}
First of all, let  $x_\ast$ be an exact solution to the   system of  Boolean equations (\ref{bes}). Then by Lemma~\ref{prop1}, there holds
$$
\mathbf{M}_{f_i}(\Theta_m(x_\ast))=\Theta_1(\sigma_i),\  \i=1,\dots,n.
$$
As a result, the system of linear equations $\mathcal{E}^{\rm b}_i: \mathbf{H}_i \mathbf{y}=\mathbf{z}_i$ with $\mathbf{H}_i = \mathbf{M}_{f_i}$ and $\mathbf{z}_i=\Theta_1(\sigma_i)$ contains at least one exact solution  $\Theta_m(x_\ast)$. This leads to $\mathcal{E}^{\rm b}=\mcap_{i=1}^n \mathcal{E}^{\rm b}_i\neq \emptyset$. Therefore, Lemma \ref{prop2} is applicable  to the {\sf DistributedLAE} algorithm, and each node $i$ indeed obtains the $\mathbf{y}_s=\sum_{i=1}^n\mathcal{P}_\ast (\bm{\beta}_{i,s})/n \in\mathcal{E}$ for $s=1,\dots,2^m+1$.

Next,  we prove that with probability one there holds ${\sf
  Aff}(\yb_1,\dots,\yb_{2^m+1})=\mathcal{E}^{\rm b}$. Since $\mathcal{E}^{\rm b}$ is an affine subspace from its definition, and $\mathbf{y}_s\in \mathcal{E}^{\rm b}$ for all $s=1,\dots,2^m+1$, we conclude that  ${\sf
  Aff}(\yb_1,\dots,\yb_{2^m+1})\subseteq\mathcal{E}$ is a sure event. We continue to show that
  ${\rm dim}({\sf
  Aff}(\yb_1,\dots,\yb_{2^m+1}))=\dim(\mathcal{E}^{\rm b})$ with probability one, where the randomness arises from the random initial values $\mathbf{x}_i(0)$ for Step 3 of the algorithm {\sf DistributedBooleanEquationSolver}.

  Note that, there holds from the basic properties of projections onto affine subspaces that
  $$
  \mathbf{y}_s=\sum_{i=1}^n\mathcal{P}_\ast (\bm{\beta}_{i,s})/n =\mathcal{P}_\ast\big(\sum_{i=1}^n \bm{\beta}_{i,s}/n\big)=  \sum_{i=1}^n  \mathbf{P}_\ast \bm{\beta}_{i,s}/n +\mathbf{a}_\ast.
$$
where $\mathbf{P}_\ast\in \mathbb{R}^{2^m \times 2^m}$ is a projection matrix, and $\mathbf{a}_\ast$ is a vector in $\mathbb{R}^{2^m}$. In fact, letting $\mathbf{e}\in\mathcal{E}$, $\mathbf{P}_\ast$ is the projector onto the linear subspace
$
\mathcal{E}^{\rm b}_\ast:=\big\{\mathbf{y}-\mathbf{e}: \ \mathbf{y}\in\mathcal{E}\big\}.
$  We thus have
\begin{align}\label{99}
(\yb_2 - \yb_1,
\ldots, \yb_{2^m+1} - \yb_1)=\mathbf{P}_\ast (\mathbf{l}_2-\mathbf{l}_1,\dots, \mathbf{l}_{2^m+1}-\mathbf{l}_1)
\end{align}
where $\mathbf{l}_j=\sum_{i=1}^n\bm{\beta}_{i,j}/n$ for $j=2,\dots,2^m+1$. Now, since the $\bm{\beta}_{i,s}\sim {\sf uniform}([0,1]^{2^m})$  are selected randomly and independently, it is trivial to see that ${\rm rank}(\mathbf{l}_2-\mathbf{l}_1,\dots, \mathbf{l}_{2^m+1}-\mathbf{l}_1)=2^m$ holds with probability one. This implies
\begin{align}
\rm{rank}(\yb_2 - \yb_1,
\ldots, \yb_{2^m+1} - \yb_1)= \rm{rank}(\mathcal{E}^{\rm b})
\end{align}
with probability one. Therefore,  ${\rm dim}({\sf
  Aff}(\yb_1,\dots,\yb_{2^m+1}))=\dim(\mathcal{E}^{\rm b})$ with probability one, leading to ${\sf
  Aff}(\yb_1,\dots,\yb_{2^m+1})=\mathcal{E}^{\rm b}$ in view of the fact ${\sf
  Aff}(\yb_1,\dots,\yb_{2^m+1})\subseteq\mathcal{E}$.

Finally, based on Lemma  \ref{le:correctnessofalgo1}, ${\sf BooleanVectorSearch}(\yb_1,\dots,\yb_{2^m+1})$ with probability one returns the set $S$ as $\mathcal{E}^{\rm b}\mcap \Delta_{2^m}$. This fact can be established from the following two aspects.
\begin{itemize}
\item[(i)] Let $x\in \Upsilon_m (\mathcal{E}^{\rm b}\mcap \Delta_{2^m})$. Then there holds $\mathbf{M}_{f_i}(\Theta_m(x))=\Theta_1(\sigma_i)$ for all $i=1,\dots,n$, or equivalently, $f_i(x)=\sigma_i$ for all $i=1,\dots,n$ based on Lemma \ref{prop1}.

\item[(ii)] Let $x\in\{0,1\}^m$ satisfying  $f_i(x)=\sigma_i$ for all $i=1,\dots,n$.  Then $\Theta_m(x)\in\Delta_{2^m}$ from the definiton of $\Theta_m(\cdot)$, and there must  also hold $\Theta_m(x)\in \mathcal{E}^{\rm b}$ again from Lemma~\ref{prop1}. As a result, we also have $x\in \Upsilon_m (\mathcal{E}^{\rm b}\mcap \Delta_{2^m})$.
\end{itemize}
We have now concluded that $\mathcal{S}=\Theta_m(S)$ is the solution set of  the system of Boolean equations (\ref{bes}) with probability one. The proof is complete.
\subsection*{B. Proof of Lemma \ref{lem5}}
In view of (\ref{eqapp}) and (\ref{eqrr}), there always holds
$$
\sum_{s=1}^{2^m+1}{\rm dist}(\widehat{\mathbf{y}_{i,s}},\mathcal{E}^{\rm b})\leq   (2^m+1) C_\ast e^{-\gamma_\ast T} =\epsilon.
$$
Therefore, the true solution space $\mathcal{E}^{\rm b}=\mcap_{k=1}^n \mathcal{E}^{\rm b}_k$ will always
be a feasible   point of (\ref{opt}).  Let us assume for the sake of building up a contradiction argument that  for any $\epsilon>0$, there exists\footnote{Technically, this $\mathcal{A}_\star$ is a random set and depends on the particular $\epsilon$. }  $\mathcal{A}_\star \in \mathfrak{A}$ as a solution to (\ref{opt}), and ${\rm rank}(\mathcal{A}_\star)< {\rm rank}(\mathcal{E}^{\rm b})$ with at least a probability $p>0$ (which does not depend on $T$ or $\epsilon$).

We have proved in the proof of Theorem \ref{thm1} that   there holds
\begin{align}\label{t2-0}
\rm{rank}(\yb_2 - \yb_1,
\ldots, \yb_{2^m+1} - \yb_1)= \rm{rank}(\mathcal{E}^{\rm b})
\end{align}
with probability one. Therefore, if indeed  $\mathcal{A}_\star \in \mathfrak{A}$ is  a solution to (\ref{opt}), there must hold
\begin{align}\label{t2-1}
\sum_{s=1}^{2^m+1}{\rm dist}({\mathbf{y}_s},\mathcal{A}_\star)&\leq    \sum_{s=1}^{2^m+1} \Big( {\rm dist}(\widehat{\mathbf{y}_{i,s}},\mathcal{A}_\star)+\|{\mathbf{y}_s}-\widehat{\mathbf{y}_{i,s}}\| \Big)\nonumber\\
&\leq \sum_{s=1}^{2^m+1}\Big( {\rm dist}(\widehat{\mathbf{y}_{i,s}},\mathcal{A}_\star)+ \|r_{i,s}(T)\| \Big)\nonumber\\
&\leq 2\epsilon
\end{align}
   where the first inequality is directly from the triangle inequality,   the second inequality is from
   (\ref{eqapp}), and the last inequality follows from (\ref{eqrr})  and (\ref{opt}).  Noting  the fact that $\bm{\beta}_{i,s}\sim {\sf uniform}([0,1]^{2^m})$  are selected randomly and independently, and the identity that for  $\mathbf{l}_j=\sum_{i=1}^n\bm{\beta}_{i,j}/n$
    $$
    (\yb_2 - \yb_1,
\ldots, \yb_{2^m+1} - \yb_1) =\mathbf{P}_\ast(\mathbf{l}_2-\mathbf{l}_1,\dots, \mathbf{l}_{2^m+1}-\mathbf{l}_1),
    $$
the probability that (\ref{t2-1}) and ${\rm rank}(\mathcal{A}_\star)< {\rm rank}(\mathcal{E}^{\rm b})$ simultaneously  holds goes to zero as $\epsilon $ tends to zero. We therefore have established a contradiction, and proved that $
\lim_{T\to \infty} \mathbb{P}\big(\mathcal{E}^{\rm b}\ \mbox{is a solution to (\ref{opt})} \big)=1
$.

\subsection*{C. Proof of Lemma \ref{lem6}}Suppose that $\mathcal{A}_\star \in  \mathfrak{A}_{i,\epsilon}^\star$ is a solution to (\ref{opt}). Then there hold
 \begin{itemize}
\item [(i)] ${\rm rank}(\mathcal{A}_\star)={\rm rank}(\mathcal{E}^{\rm b})$;

\item [(ii)] $\sum_{s=1}^{2^m+1}{\rm dist}({\mathbf{y}_s},\mathcal{A}_\star) \leq 2\epsilon$.
 \end{itemize}
We prove the desired lemma by establishing $\big(\mathcal{A}_\star \mcap \Delta_{2^m}\big) \subseteq (\mathcal{E}^{\rm b}\mcap \Delta_{2^m})$ with probability one when $\epsilon$ is sufficiently small.

Note that the ${\mathbf{y}_s}$ are distributed over the convex polyhedron   $\mathcal{P}_\ast\big([0,1]^{2^m}\big)\subseteq\mathcal{E}$ due to the fact that $\mathbf{y}_s=  \sum_{i=1}^n  \mathbf{P}_\ast \bm{\beta}_{i,s}/n +\mathbf{a}_\ast$. The condition (i) and (ii) imply that for sufficiently small $\epsilon$, there holds
$$
\mathsf{Aff}(\mathcal{P}_{\mathcal{A}_\star}(\mathbf{y}_1),\dots, \mathcal{P}_{\mathcal{A}_\star}(\mathbf{y}_{2^m+1}))=\mathcal{A}_\star
$$ with probability one since $\mathsf{Aff}(\mathbf{y}_1,\dots, \mathbf{y}_{2^m+1})=\mathcal{E}^{\rm b}$
with probability one from the proof of Theorem \ref{thm1}. Let $\bm{\delta}\in  \mathcal{A}_\star \mcap \Delta_{2^m}$. Then    there are $\lambda_k,k=1,\dots,2^m+1$ which are  upper bounded by some absolute constant $B>0$,  so that
\begin{align}
\bm{\delta}=\sum_{k=1}^{2^m+1} \lambda_k \mathcal{P}_{\mathcal{A}_\star}(\mathbf{y}_k).
\end{align}
This implies
\begin{align}\label{p2}
{\rm dist}(\bm{\delta}, \mathcal{E}^{\rm b})&\leq  B \sum_{k=1}^{2^m+1} {\rm dist}(\mathcal{P}_{\mathcal{A}_\star}(\mathbf{y}_k), \mathcal{E}^{\rm b})\nonumber\\
&\leq  B \sum_{k=1}^{2^m+1} {\rm dist}(\mathcal{P}_{\mathcal{A}_\star}(\mathbf{y}_k), \mathbf{y}_k)\nonumber\\
&=  B \sum_{k=1}^{2^m+1} {\rm dist}( \mathbf{y}_k,\mathcal{A}_\star) \nonumber\\
&\leq 2(2^m+1)B\epsilon.
\end{align}

Now,  (\ref{p2}) suggests that
$$
\big\|\mathbf{M}_{f_i}(\bm{\delta})-\Theta_1(\sigma_i)\big\|\leq 2(2^m+1)B\epsilon,\  i=1,\dots,n.
$$
The only possibility for this hold with small enough $\epsilon$ is
$$
f_i(\Upsilon_m(\bm{\delta}))=\sigma_i, i=1,\dots,n
$$
because $\mathbf{M}_{f_i}(\bm{\delta})= \Theta_1 f_i(\Upsilon_m(\bm{\delta}))$ by Lemma \ref{prop1}. As a result, $\bm{\delta}\in\mathcal{E}\mcap \Delta_{2^m}$, which implies  $\big(\mathcal{A}_\star \mcap \Delta_{2^m}\big) \subseteq (\mathcal{E}^{\rm b}\mcap \Delta_{2^m})$ with probability one.  Therefore, $\mathcal{E}^{\rm b}$ must be  the unique solution to (\ref{opt2})  if indeed $\mathcal{E}^{\rm b}\in   \mathfrak{A}_{i,\epsilon}^\star$.  By Lemma  \ref{lem5}, there holds  $$
\lim_{T\to \infty} \mathbb{P}\big(\mathcal{E}^{\rm b}\in   \mathfrak{A}_{i,\epsilon}^\star\big)=1
$$
and this concludes the proof of the desired lemma.

\subsection*{D. Proof of Theorem \ref{thm3}}

If $\chi_0$ is the cardinality of the image of the Boolean mapping $[f_1\ \dots\ f_n]^\top$, then the
  matrix $\mathbf{M}_f$
 has at most $\chi_0$ distinct columns. Since each column of  $\mathbf{M}_f$ is a vector in $\mathbb{R}^{2n}$ with the form of
 \[
 \begin{bmatrix} v_1\\
 \vdots \\
 v_n
 \end{bmatrix}
 \]
 where $v_k\in\{(1 \ 0)^\top, (0 \ 1)^\top\}$, we conclude that
  \[
 {\rm rank} (\mathbf{M}_f) \leq \chi_0.
 \]
 This implies  ${\rm dim}(\mathcal{E}^{\rm b})\geq d-\chi_0$. Repeating the proof of Theorem \ref{thm1} we know that
\begin{align}
\rm{rank}(\yb_2 - \yb_1,
\ldots, \yb_{k_\ast} - \yb_1)= \rm{dim}(\mathcal{E}^{\rm b})
\end{align}
with probability one with $k_\ast=d-\chi_0+1$. The desired theorem thus follows from the same argument as the proof of Theorem \ref{thm1}.
\subsection*{E. Proof of Lemma \ref{lem7}}
Denote $\mathbf{x}(t)=(\mathbf{x}_1^\top(t)\ \dots \ \mathbf{x}_n^\top(t))^\top$. Let $\mathbf{A}^\dag$ be the  M-P pseudoinverse of a matrix $\mathbf{A}$. Let $\rho(\mathbf{A})$ and $\sigma(\mathbf{A})$ represent the spectral radius and spectrum of a matrix $\mathbf{A}$, respectively.  Then from the basic representations  of the affine projections $\mathcal{P}_i$, the algorithm (\ref{eq2}) can be written in a compact form as
\begin{align}\label{r2}
\mathbf{x}(t+1)= \mathbf{P} \mathbf{W}\mathbf{x}(t)+ \mathbf{b}
\end{align}
where $\mathbf{P}={\rm diag}(\mathbf{P}_1\ \dots\ \mathbf{P}_n)$ is a block diagonal matrix with $\mathbf{P}_i=\mathbf{I}_d - \mathbf{H}_i^\dag \mathbf{H}_i$ being a projection matrix, $\mathbf{W}=W\otimes \mathbf{I}_d$ with $W\in\mathbb{R}^{n\times n}$ defined in (\ref{3}), and
\[
\mathbf{b}=\begin{bmatrix}
\mathbf{H}_1^\dag  \mathbf{z}_1\\
\mathbf{H}_2^\dag  \mathbf{z}_2\\
\vdots\\
\mathbf{H}_n^\dag  \mathbf{z}_n
\end{bmatrix}.
\]
Note that, the matrix $\mathbf{P}_i$ is the projector onto the linear subspace $\mathcal{L}_i:=\{\mathbf{y}:\mathbf{H}_i \mathbf{y}=0\}$. For any vector $\mathbf{u}=(\mathbf{u}_1^\top \ \dots \ \mathbf{u}_n^\top )^\top$ with $\mathbf{u}_i\in \mathbb{R}^d$, there holds
\begin{align}
\big\| \mathbf{P} \mathbf{W} \mathbf{u}\big\|^2&=\sum_{i=1}^n \big\|\mathbf{P}_i \sum_{j=1}^n W_{ij} \mathbf{u}_j\big\|^2 \nonumber\\
&\leq \sum_{i=1}^n\sum_{j=1}^n W_{ij} \|\mathbf{u}_j\|^2\nonumber\\
&=\sum_{j=1}^n   \|\mathbf{u}_j\|^2\nonumber\\
&= \|\mathbf{u}\|^2.
\end{align}
This implies  $\|\mathbf{P}\mathbf{W}\|_2\leq 1$, and consequently, we have $\rho(\mathbf{P} \mathbf{W})\leq 1$. Moreover, all eigenvalues of $\mathbf{P} \mathbf{W}$ on the unit circle of the complex plain must have equal algebraic and geometric multiplicities. We proceed to establish the following claims on the matrix $\mathbf{P}\mathbf{W}$.

\noindent{\it Claim A.} $|\lambda_i(\mathbf{P}\mathbf{W})|<1$ if for all $\lambda_i(\mathbf{P}\mathbf{W}) \neq 1\in \sigma(\mathbf{P}\mathbf{W})$.

\noindent{\it Claim B.} If $1\in \sigma(\mathbf{P}\mathbf{W})$, then the eigenspace corresponding to the eigenvalue $1$ is a subspace of $\mathcal{M}:\{\mathbf{1}_m\otimes \mathbf{y}: \mathbf{y}\in\mathbb{R}^d\}$.

In fact, let us consider the linear dynamical system
\begin{align}\label{r1}
\bar{\mathbf{x}}(t+1)=   \mathbf{P} \mathbf{W}\bar{\mathbf{x}}(t),
\end{align}
which defines a special projection consensus algorithm in the following form:
 \begin{align}
\bar{\mathbf{x}}_i(t+1)=  \bar{\mathcal{P}}_i \Big( \sum_{j=1}^n W_{ij} \bar{\mathbf{x}}_j(t)\Big), \ i=1,\dots,n.
\end{align}
Here each $\bar{\mathcal{P}}_i$ is the projection onto the linear subspace $\mathcal{L}_i$. As the $\mathcal{L}_i$'s are linear subspaces, there always holds $\mcap_{i=1}^n \mathcal {L}_i \neq \emptyset$. Therefore,  we can directly invoke  Lemma 3 of \cite{nedic2010constrainedTAC} to conclude that along (\ref{r1}), all $\bar{\mathbf{x}}_i(t)$ converges to a common static value for all initial conditions at time grows to infinity. Since (\ref{r1}) is a linear time-invariant system, the above two  claims must hold.

\noindent Proof of (i). We divide the proof into two cases.
\begin{itemize}
\item[(a).] Suppose $\rho(\mathbf{P} \mathbf{W})<1$. Then with $\mathbf{y}_\ast= (\mathbf{I}_{nd} -\mathbf{P} \mathbf{W})^{-1} \mathbf {b}$, (\ref{r2}) becomes
\begin{align}
 \mathbf{x}(t+1)- \mathbf{y}_\ast=   \mathbf{P} \mathbf{W}({\mathbf{x}}(t)-\mathbf{y}_\ast).
\end{align}
Obviously there holds $\lim_{t\to \infty}\mathbf{x}(t)=\mathbf{y}_\ast$.
\item[(b).] Suppose $\rho(\mathbf{P} \mathbf{W})=1$. From  Claim A, we know that the only eigenvalue of  $\mathbf{P} \mathbf{W}$ with magnitude one is $1$, which has equal algebraic and geometric multiplicity. Thus, we can find a real orthogonal matrix $\mathbf{T}$ such that
\begin{align}
\mathbf{T}^{-1} \mathbf{P} \mathbf{W} \mathbf{T}=\begin{bmatrix}
\mathbf{I}_c & 0 \\
0 & \bar{\mathbf{P}}_W
\end{bmatrix}
\end{align}
where $c$ is the multiplicity of the eigenvalue one, and all eigenvalues of $\bar{\mathbf{P}}_W$ are strictly within the unit circle. Letting $\mathbf{y}(t)=\mathbf{T}\mathbf{x}(t)$,(\ref{r2}) is written as
\begin{align} \label{r3}
 \mathbf{y}(t+1)=  \mathbf{T}^{-1} \mathbf{P} \mathbf{W} \mathbf{T}{\mathbf{y}}(t)+ \mathbf{T}^{-1} \mathbf{b}.
\end{align}
The first $c$ columns of the matrix $\mathbf{T}$,  $\mathbf{T}_1,\dots,\mathbf{T}_c$, are eigenvectors of the matrix $\mathbf{P} \mathbf{W}$ corresponding to eigenvalue one. We write
$$
\mathbf{T}=(\mathbf{T}_1\ \dots\ \mathbf{T}_c\ \mathbf{T}_\ast).
$$

Now, from Claim B, each $\mathbf{T}_k$ can be written into $\mathbf{1}_n\otimes \mathbf{h}_k$ with $\mathbf{h}_k\in \mathbb{R}^d$. This implies
$$
\mathbf{P}_i\mathbf{h}_k=(\mathbf{I}_d - \mathbf{H}_i^\dag \mathbf{H}_i)\mathbf{h}_k=\mathbf{h}_k
$$
 for all $i=1,\dots,n$ and all $k=1,\dots,c$. As a result, $\mathbf{H}_i^\dag \mathbf{H}_i \mathbf{h}_k=0$ for all $i=1,\dots,n$ and all $k=1,\dots,c$, which further implies
 \begin{align}
 \mathbf{T}_k^\top \mathbf{b}=  \mathbf{T}_k^\top \begin{bmatrix}
\mathbf{H}_1^\dag  \mathbf{z}_1\\
\mathbf{H}_2^\dag  \mathbf{z}_2\\
\vdots\\
\mathbf{H}_n^\dag  \mathbf{z}_n
\end{bmatrix}=\sum_{i=1}^n \mathbf{h}_k^\top \mathbf{H}_i^\dag  \mathbf{z}_i=0,\ \ k=1,\dots,c
 \end{align}
 since $\mathbf{H}_i^\dag \mathbf{H}_i \mathbf{h}_k=0$ implies $\mathbf{h}_k^\top \mathbf{H}_i$ utilizing the basic properties of   M-P pseudoinverse. The system (\ref{r3}) can therefore be further written as
\begin{equation}
 \begin{aligned}
 \mathbf{y}_a(t+1)&=  {\mathbf{y}}_a(t)\\
 \mathbf{y}_b(t+1)&=\bar{\mathbf{P}}_W \mathbf{y}_b(t)+\bar{b} \label{r4}
\end{aligned}
\end{equation}
where ${\mathbf{y}}_a(t)$ consists of the first $c$ entries of ${\mathbf{y}}(t)$, and ${\mathbf{y}}_b(t)$ has the remaining entries of  ${\mathbf{y}}(t)$. From (\ref{r4}), the case has been reduced to Case (a), and each $\mathbf{x}_i(t)$ must converge to a static value  because $\mathbf{y}(t)$ does.
\end{itemize}
This concludes the proof of statement (i).

\medskip

\noindent{Proof of (ii).} Suppose  at the limits of the $\mathbf{x}_i(t)$ there holds $\mathbf{y}_1^\ast=\dots= \mathbf{y}_n^\ast=\mathbf{u}^\ast$. This means $\mathbf{u}^\ast \in \mathcal{E}_i$ for all $i$, and thus $\mcap_{i=1}^n \mathcal{E}_i \neq \emptyset$. Consequently,  the system of linear equations (\ref{LAE}) admits at least one exact solution, which is a contradiction with our standing assumption of the lemma.  There must exist at least two nodes $j,k\in\mathrm{V}$ such that $\mathbf{y}_j^\ast\neq \mathbf{y}_k^\ast$, and this concludes the proof.

 \subsection*{F. Proof of Theorem \ref{thm4}}
 Suppose the system of Boolean equation (\ref{bes}) is satisfiable. Then the linear equation $
    \mathbf{M}_{f_i} \mathbf{y}=\Theta_1(\sigma_i),\ \ i=1,\dots,n
    $
    admits at least one exact solution. Therefore, applying Lemma \ref{prop2} we conclude that in Step 4 of the algorithm \ref{algo:solbability}, there holds $\tilde{\mathbf{y}}_i=\sum_{k=1}^n\mathcal{P}_\ast (\mathbf{x}_k(0))/n$ for all $i=1,\dots,n$. This obviously leads to $\tilde{\mathbf{y}}_{\rm ave}=\tilde{\mathbf{y}}_i$ for all $i$. As a result,  Algorithm \ref{algo:solbability} proceeds to Step $6$ and Step $7$. Based on Theorem \ref{thm1},  with probability one $\mathcal{S}$ will be returned as the exact solution set of the Boolean equations (\ref{bes}), which certainly satisfies $\mathcal{S}\neq\emptyset$. Therefore, Algorithm \ref{algo:solbability} correctly returns {\sf satisfiable}.

   Now suppose on the other hand   the system of Boolean equation (\ref{bes}) is not satisfiable. There will be two cases.
   \begin{itemize}
   \item[(a)] Let the linear equation $
    \mathbf{M}_{f_i} \mathbf{y}=\Theta_1(\sigma_i),\ \ i=1,\dots,n
    $ admit no exact solutions. Based on Lemma \ref{lem7}, we know that the Step 4 of the algorithm \ref{algo:solbability} does produce a finite value $\tilde{\mathbf{y}}_i$ at each node $i$
  as the limit of the algorithm (\ref{eq2}), but there exist at least two nodes $j$ and $k$ such that $\tilde{\mathbf{y}}_j\neq \tilde{\mathbf{y}}_k$. As a result, except for a set of initial values with measure zero, there holds $\tilde{\mathbf{y}}_{\rm ave}=\tilde{\mathbf{y}}_i$ for any node $i$. Consequently, with probability one, Algorithm \ref{algo:solbability} correctly returns {\sf unsatisfiable} at Step $5$.

  \item[(b)] Let the linear equation $
    \mathbf{M}_{f_i} \mathbf{y}=\Theta_1(\sigma_i),\ \ i=1,\dots,n
    $ admit at least one exact solutions. Again,  from Theorem \ref{thm1},  with probability one $\mathcal{S}$ will be returned as the exact solution set of the Boolean equations (\ref{bes}), in which case there must hold $\mathcal{S}=\emptyset$. Therefore, Algorithm \ref{algo:solbability} correctly returns {\sf unsatisfiable} at Step $7$.
  \end{itemize}
  We have now proved the desired theorem.

 \subsection*{G. Proof of Theorem \ref{thm-DP}}
  Let $f(x)=\sigma$ and $f'(x)=\sigma'$ be two Boolean equation systems in the form of (\ref{bes}) that are adjacent, i.e.,      there is a unique $k\in\{1,\dots,n\}$ such that $\mathcal{S}_k$ and $\mathcal{S}_k'$  with $|\mathcal{S}_k|=|\mathcal{S}_k'|$ differ by only one entry, and the rest $\mathcal{S}_j$ and $\mathcal{S}_j'$
are identical. Here $\mathcal{S}=\mathcal{S}_1\times \dots \times \mathcal{S}_n,\mathcal{S}'=\mathcal{S}_1'\times \dots \times \mathcal{S}_n'$ as their (local) solution spaces, respectively.

 Denote        $\mathbf{H}_i \mathbf{y}= \mathbf{z}_i $ and  $\mathbf{H}_i' \mathbf{y}= \mathbf{z}_i $ as the induced algebraic equation for  $f_i(x)=\sigma_i$ and $f_i'(x)=\sigma_i$, respectively. The following lemma holds.
  \begin{lemma}\label{lemma-8}There hold $\mathbf{H}_i =\mathbf{H}_i' $ for all $i\neq k\in \{1,\dots,n\}$, and  $\mathbf{H}_k,\mathbf{H}_k' $ differ by exactly two columns.
  \end{lemma}
\noindent {\it Proof.} Note that, for the two Boolean equations $f_i(x)=\sigma_i$ and $f_i'(x)=\sigma_i$, the matrix representation of $f_i$ and $f_i'$ are
 \begin{align}
\mathbf{H}_i=\Big[\delta_{2}^{f_i(\itob{1})}\ \delta_{2}^{f_i(\itob{2})}\ \dots, \ \delta_{2}^{f_i(\itob{2^m})}\Big]
\end{align}
and
 \begin{align}
\mathbf{H}_i'=\Big[\delta_{2}^{f_i'(\itob{1})}\ \delta_{2}^{f_i(\itob{2})}\ \dots, \ \delta_{2}^{f_i'(\itob{2^m})}\Big].
\end{align}
For the two matrices $\mathbf{H}$ and $\mathbf{H}'$, the columns that are identical to $\Theta_1(\sigma_i)$, corresponds to  solutions of the Boolean equation.  Now, from the definition of adjacency, there hold  $|\mathcal{S}_k|=|\mathcal{S}_k'|$, and $\mathcal{S}_k$ and $\mathcal{S}_k'$ differ by only one entry. Consequently, $\mathbf{H}_i$ and $\mathbf{H}_i'$
 differ by exactly two columns. \hfill$\square$

%%%%%%%%%%%%%%%%%%%%%%%%%%%%
\begin{lemma}
 There hold $\|\mathbf{H}_k^\dag \mathbf{H}_k - {\mathbf{H}_k'}^\dag \mathbf{H}_k'\|\leq 2+6\times2^{{(m/2)}}$ and $\|\mathbf{H}_k^\dag  - {\mathbf{H}_k'}^\dag\|\leq 6$.
\end{lemma}
\noindent {\it Proof.}
We observe that (i) if $\rank(\mathbf{H}_k)=1$, then $\mathbf{H}_k^\dag = 2^{-m} \mathbf{H}_k^\top$, which yields $\|\mathbf{H}_k^\dag\| = 2^{-m/2}$; (ii) if $\rank(\mathbf{H}_k)=2$, then $\mathbf{H}_k^\dag = \mathbf{H}_k^\top\big(\mathbf{H}_k\mathbf{H}_k^\top\big)^{-1}$, which yields $\|\mathbf{H}_k^\dag\| = \sigma_{M}\left(\big(\mathbf{H}_k\mathbf{H}_k^\top\big)^{-1}\right)\leq 1$. Thus, we have $\|\mathbf{H}_k^\dag\|\leq 1$. Using similar arguments, we have $\|{\mathbf{H}_k'}^\dag\|\leq 1$. Then with Lemma \ref{lemma-8} and according to \cite[Theorem 10.4.5]{Campbell2009}, we have
\[\begin{array}{rcl}
\|\mathbf{H}_k^\dag  - {\mathbf{H}_k'}^\dag\| &\leq& 3\max\{\|\mathbf{H}_k^\dag\|^2,\|{\mathbf{H}_k'}^\dag\|^2\}\|\mathbf{H}_k  - {\mathbf{H}_k'}\| \\
&\leq& 3\|\mathbf{H}_k  - {\mathbf{H}_k'}\|\\
 &\leq& 6\,
\end{array}\]
and
 \[\begin{array}{rcl}
 \|\mathbf{H}_k^\dag \mathbf{H}_k - {\mathbf{H}_k'}^\dag \mathbf{H}_k'\| &\leq& \|{\mathbf{H}_k^\dag} (\mathbf{H}_k - \mathbf{H}_k')\|+ \|(\mathbf{H}_k^\dag - {\mathbf{H}_k'}^\dag) \mathbf{H}_k'\|\, \\
 &\leq& \|{\mathbf{H}_k^\dag} \|\|\mathbf{H}_k - \mathbf{H}_k'\|+ \|\mathbf{H}_k^\dag - {\mathbf{H}_k'}^\dag\|\| \mathbf{H}_k'\|\, \\
 &\leq& 2 \| \mathbf{H}_k^\dag\|+6\| \mathbf{H}_k'\|\\
 &\leq& 2+6\times2^{{(m/2)}}\,.
 \end{array}\]
 \hfill$\square$
%%%%%%%%%%%%%%%%%%%%%%%%%

Now we consider the distributed linear equation solver (\ref{newlae}) applied to $\mathbf{H} \mathbf{y}= \mathbf{z}  $ and  $\mathbf{H}' \mathbf{y}= \mathbf{z}  $, with initial value taken within $[-1,1]$   for each entry of $\mathbf{x}_i(0)$.  Denote $\mathrm{M}_\ast$ as the mapping from $(\mathbf{H},\mathbf{z})$ (and $(\mathbf{H}',\mathbf{z}')$, with slight abuse of notation) to $\mathbf{x}(k), k=1,\dots,T$. Let $\mathbf{P}'= {\rm diag}(\mathbf{P}_1'\ \dots\ \mathbf{P}_n')$ is a block diagonal matrix with $\mathbf{P}_i'=\mathbf{I}_d - {\mathbf{H}_i'}^\dag {\mathbf{H}_i'}$, and
\[
\mathbf{b}'=\begin{bmatrix}
{\mathbf{H}_1'}^\dag  \mathbf{z}_1\\
{\mathbf{H}_2'}^\dag  \mathbf{z}_2\\
\vdots\\
{\mathbf{H}_n'}^\dag  \mathbf{z}_n
\end{bmatrix}.
\]

\begin{lemma}
Let $\kappa = 6+(2+6\times2^{{(m/2)}})2^{m}$. For any $\epsilon>0$, if $\sigma \geq\kappa T/ \epsilon$, then there holds
\begin{align}\label{eq:lemma-dp}
\mathbb{P}(\mathrm{M}_\ast(\mathbf{H},\mathbf{z})\in \mathfrak{E})\leq e^\epsilon  \mathbb{P}(\mathrm{M}_\ast(\mathbf{H}',\mathbf{z})\in \mathfrak{E})
\end{align}
for all events $\mathbf{x}(k),k=1,\dots,T$.
\end{lemma}
\noindent {\it Proof.}
  Define
  \[
  \hat \xb(t) := \mathbf{P} \mathbf{W}\mathbf{x}(t)+ \mathbf{b} + \omegab(t)\,.
  \]
  It is clear that $\xb(t) = \Pi_{\Omega}\big(\hat \xb(t)\big)$ where $\Pi_{\Omega}(\cdot)$ is the deterministic projector onto the set $\Omega$.
  Denote the mapping from $(\mathbf{H},\mathbf{z})$ to $\big(\xb(t)\big)_{t=1}^T$ by $\widehat{\mathrm{M}}_\ast(\mathbf{H},\mathbf{z})$. Thus, it can be easily seen that there is a deterministic mapping $\widehat\Pi_{\Omega}(\cdot)$ such that $\mathrm{M}_\ast(\mathbf{H},\mathbf{z}) = \widehat\Pi_{\Omega}\Big(\widehat{\mathrm{M}}_\ast(\mathbf{H},\mathbf{z})\Big)$ holds. Recalling the fact that the differential privacy is immune from the post-processing \cite{dwork2014algorithmic}, we can find that (\ref{eq:lemma-dp}) is satisfied if and only if
  \begin{align}\label{eq:lemma-dp-2}
\mathbb{P}(\widehat{\mathrm{M}}_\ast(\mathbf{H},\mathbf{z})\in \mathcal{E})\leq e^\epsilon  \mathbb{P}(\widehat{\mathrm{M}}_\ast(\mathbf{H}',\mathbf{z})\in \mathcal{E})
\end{align}
is satisfied. Therefore, in the following we focus on proving (\ref{eq:lemma-dp-2}).

Given any $\xb(t)\in\Omega$, we have
\begin{equation}\label{eq:pdf}
\begin{array}{rcl}
  \frac{\mathbb{P}\big(\mathbf{P} \mathbf{W}\mathbf{x}(t)+ \mathbf{b} + \omegab(t)\in R_t\big)}{\mathbb{P}\big(\mathbf{P}' \mathbf{W}\mathbf{x}(t)+ \mathbf{b}' + \omegab(t)\in R_t\big)} &\leq& \mbox{exp}\Big(\frac{\|(\mathbf{P}-\mathbf{P}') \mathbf{W}\mathbf{x}(t) +(\mathbf{b}-\mathbf{b}')\|_1}{\sigma} \Big)\,\\
   &\leq& \mbox{exp}\Big(\frac{\|(\mathbf{P}-\mathbf{P}') \mathbf{W}\mathbf{x}(t) \|_1}{\sigma} +\frac{\|\mathbf{b}-\mathbf{b}'\|_1}{\sigma}\Big)
\end{array}
\end{equation}
for all $R_t\subset \mathbb{R}^{2^m}$.
Since $\xb(t)\in\Omega$ for all $t$, we note that
\[\begin{array}{rcl}
\|(\mathbf{P}-\mathbf{P}') \mathbf{W}\mathbf{x}(t)\|_1 &=& \|(\mathbf{H}_k^\dag \mathbf{H}_k - {\mathbf{H}_k'}^\dag \mathbf{H}_k')\sum_{j=1}^{n}w_{kj}\mathbf{x}_i(t)\|_1\\
&\leq& \|\mathbf{H}_k^\dag \mathbf{H}_k - {\mathbf{H}_k'}^\dag \mathbf{H}_k'\|\|\sum_{j=1}^{n}w_{kj}\mathbf{x}_i(t)\|_1\\
&\leq& (2+6\times2^{{(m/2)}})\|\sum_{j=1}^{n}w_{kj}\mathbf{x}_i(t)\|_1\\
&\leq& (2+6\times2^{{(m/2)}})2^{m}
\end{array}\]
and
\[\begin{array}{rcl}
\|\mathbf{b}-\mathbf{b}'\|_1 &=& \|(\mathbf{H}_k^\dag - {\mathbf{H}_k'}^\dag )\zb_k\|_1\,\\
&\leq& \|\mathbf{H}_k^\dag - {\mathbf{H}_k'}^\dag\|\|\zb_k\|_1\,\\
&\leq& 6\,.
\end{array}\]
Thus, the (\ref{eq:pdf}) can be bounded by
\begin{equation}\label{eq:pdf-2}
  \frac{\mathbb{P}\big(\mathbf{P} \mathbf{W}\mathbf{x}(t)+ \mathbf{b} + \omegab(t)\in R_t\big)}{\mathbb{P}\big(\mathbf{P}' \mathbf{W}\mathbf{x}(t)+ \mathbf{b}' + \omegab(t)\in R_t\big)}
   \leq \mbox{exp}\Big(\frac{6+(2+6\times2^{{(m/2)}})2^{m}}{\sigma}\Big):= \mbox{exp}(\kappa / \sigma)
\end{equation}
Therefore, with $\sigma \geq\kappa T/ \epsilon$, we have
\[
\frac{\mathbb{P}(\mathrm{M}_\ast(\mathbf{H},\mathbf{z})\in \mathcal{E})}{\mathbb{P}(\mathrm{M}_\ast(\mathbf{H}',\mathbf{z})\in \mathcal{E})} \leq \prod_{t=1}^{T} \frac{\mathbb{P}\big(\mathbf{P} \mathbf{W}\mathbf{x}(t)+ \mathbf{b} + \omegab(t)\in R_t\big)}{\mathbb{P}\big(\mathbf{P}' \mathbf{W}\mathbf{x}(t)+ \mathbf{b}' + \omegab(t)\in R_t\big)} \leq \mbox{exp}(T\kappa / \sigma) \leq e^\epsilon\,.
\]
This proves the desired lemma.
 \hfill$\square$

We are now in a place to prove Theorem 5. For two adjacent Boolean equations, the induced algebraic equations satisfy    $\mathbf{H}_i =\mathbf{H}_i' $ for all $i\neq k\in \{1,\dots,n\}$, and  $\mathbf{H}_k,\mathbf{H}_k' $ differs by exactly two columns. The mapping from the Boolean equation to the algebraic equation is deterministic and reversible under (\ref{eq1}), Lemma 9 and Lemma 10 guarantee  that if $\sigma \geq\kappa T/ \epsilon$, the   distributed linear equation solver (\ref{newlae}) for one round to generate a single $\mathbf{y}_{i,s}$ in Algorithm 5 is differentially private. Based on the privacy composition law of randomized algorithms (Theorem 3.14 in \cite{dwork2014algorithmic}), the desired theorem follows immediately due to the fact that Algorithm 5 contains $K$ independent rounds of the mechanism $ {\mathrm{M}}_\ast$.


\begin{thebibliography}{10}
	
\bibitem{complexitybook} S. Arora and B. Barak. {\em Computational Complexity}.  Cambridge University Press, 2009.

\bibitem{corblin2007}	F. Corblin, L. Bordeaux, Y. Hamadi, E. Fanchon, and L. Trilling, ``A SAT-based approach
to decipher gene regulatory networks,'' In Integrative Post-Genomics, RIAMS, Lyon,
2007.

\bibitem{bardet2013}M. Bardet, J.-C. Faugere, B. Salvy, and
P.-J. Spaenlehauer, ``On the complexity of solving quadratic Boolean systems,'' {\em Journal of Complexity}, vol. 29, pp, 53--75, 2013.

\bibitem{virus}P. V. Mieghem, J. Omic, and R. Kooij, ``Virus spread in networks,” {\em
	IEEE/ACM Trans. Networking}, vol. 17, no. 1, pp. 1–14, Feb. 2009.

\bibitem{social}B. Doerr, M. Fouz, and T. Friedrich, ``Why rumors spread so quickly in
social networks?” {\em Commun. ACM}, vol. 55, no. 6, pp. 70–75, 2012.

\bibitem{eco}  M. P. Keane,  P. E. Todd, and K. I. Wolpin, ``The structural estimation of behavioral models: Discrete choice dynamic programming methods and applications," Volume 4 of Handbook of Labor Economics, Chapter 4. (Elsevier). pp. 331--461, 2011.

\bibitem{ara1} J. Aracena, J. Demongeot, and E.  Goles, ``Positive and negative circuits
in discrete neural networks,'' {\em IEEE Trans. Neural Networks}, vol. 15,  pp. 77--83, 2004.

\bibitem{ara2} J. Aracena, A. Richard,  and L. Salinas, ``Fixed points in conjunctive
networks and maximal independent sets in graph contractions,'' {\em Journal of
	Computer and System Sciences} vol. 88, pp. 145--163, 2017.

\bibitem{Chaves2013} L. Tournier and M. Chaves, ``Interconnection of asynchronous Boolean networks, asymptotic and transient dynamics," {\em  Automatica}, 49(4), pp. 884-893, 2013.

\bibitem{dimitri2019}D. Katselis, C. L. Beck, and R. Srikant, ``Mixing times and structural inference
for Bernoulli autoregressive processes,'' {\em IEEE Trans. Network Science and Engineering}, vol.6, no.3, pp. 364--378, 2019.

\bibitem{nedic2010constrainedTAC}
A.~Nedic, A.~Ozdaglar and P.~A.~Parrilo,
``Constrained consensus and optimization in multi-agent networks,''
\emph{IEEE Trans. Autom. Control},	vol.~55,	no.~4,
pp.~922--938, 2010.

\bibitem{mou2015}	S. Mou, J. Liu, and A. S. Morse, ``A distributed algorithm for solving
a linear algebraic equation,'' 		\emph{IEEE Trans. Autom. Control}, 60, 11, 2863--
2878, 2015.

\bibitem{shi2017networkTAC}
G.~Shi, B.~D.~Anderson and U.~Helmke,
``Network flows that solve linear equations,''
\emph{IEEE Trans. Autom. Control},
vol.~62,
no.~6,
pp.~2659--2674, 2017.

\bibitem{2010} A. G. Dimakis, S. Kar, J. M. F. Moura, M. G. Rabbat, and A. Scaglione,
``Gossip algorithms for distributed signal processing,'' {\em Proc. IEEE}, vol. 98,
no. 11, pp. 1847-1864, 2010.

\bibitem{kar2012tit}		S. Kar, J. M. F. Moura, and K. Ramanan,``Distributed parameter estimation in sensor networks: Nonlinear observation models and imperfect
communication,” {\em IEEE Trans. Inform. Theory}, vol. 58, no. 6, pp. 3575--52,
2012.

\bibitem{bianchi2013tit}P. Bianchi, G. Fort, and W. Hachem, ``Performance of a distributed stochastic
approximation algorithm,''  {\em IEEE Transactions on Information Theory}, 59(11): 7405-7418,   2013.

\bibitem{tsi1986} J. Tsitsiklis, D. Bertsekas, and M. Athans, ``Distributed asynchronous deterministic and stochastic gradient optimization algorithms,'' {\em IEEE Trans.
	Autom. Control}, vol. 31, no. 9, pp. 803–812, 1986.

\bibitem{kar2010} S. Kar and J. M. F. Moura, ``Distributed consensus algorithms in sensor
networks: Quantized data and random link failures,'' {\em IEEE Trans. Signal
	Process.}, vol. 58, no. 3, pp. 1383–1400, Mar. 2010.

\bibitem{tit2010} T. C. Aysal and K. E. Barner, ``Convergence of consensus models
with stochastic disturbances,'' {\em IEEE Transactions on Information Theory}, vol. 56, no. 8,
pp. 4101–4113, Aug. 2010.

\bibitem{shi2015tit} G. Shi, B. D. O. Anderson, and K. H. Johansson, ``Consensus over random graph processes: network Borel-Cantelli lemmas for almost sure convergence,'' {\em IEEE Transactions on Information Theory}, 61(10): 5690-5707,   2015.

\bibitem{julien}J. M. Hendrickx, A. Olshevsky, and J.  N. Tsitsiklis, ``Distributed anonymous discrete
function computation,'' {\em  IEEE Trans. Automatic Control}, vol. 56, no. 10, pp.2276--2289, 2011.

\bibitem{psat1} R. Nieuwenhuis, A. Oliveras, and C. Tinelli, ``Solving SAT and SAT modulo theories:
From an abstract Davis–Putnam–Logemann–Loveland procedure to DPLL,'' {\em J. ACM},
53(6):937--977, 2006.

\bibitem{psat2} Y. Hamadi,  S. Jabbour, and L. Sais, ``ManySAT: a parallel SAT solver,'' {\em Journal on Satisfiability, Boolean Modeling and Computation}, vol.  6, pp. 245--262, 2009.

\bibitem{psat3} R. Martins, V. Manquinho, and  I. Lynce, ``An overview of parallel SAT solving,''  {\em Constraints} 17:304–347, 2012.

\bibitem{Cheng2009} D. Cheng and H. Qi, ``Controllability and observability of Boolean control networks," {\em Automatica},  45: 1659-1667, 2009.

\bibitem{Cheng-Qi2010} D. Cheng and H. Qi, ``A linear representation of dynamics of Boolean networks,"  {\em IEEE Transactions on  Automatic Control},  55:
2251-2258, 2010.

\bibitem{book2011} D. Cheng, H. Qi,  and Z. Li.  {\em Analysis and Control of Boolean networks:
	A Semi-tensor Product Approach.} London: Springer Verlag, 2011.

\bibitem{apa} H. H. Bauschke and J. M. Borwein, ``On projection algorithms for solving
convex feasibility problems,'' {\em  SIAM Rev.}, vol. 38, no. 3, pp. 367–426,
1996.

\bibitem{magnusbook}
M. Mesbahi and M. Egerstedt. \emph{{G}raph {T}heoretic {M}ethods in {M}ultiagent {N}etworks}.
\newblock Princeton University Press,
2010.

\bibitem{yangtao}T. Yang, J. George, J. Qin, X. Yi, and  J. Wu, ``Distributed least squares solver for network linear equations," {\em Automatica}, vol. 113, no. 108798, 2020.


\bibitem{tomlin20}R. Dobbe, Y. Pu, J. Zhu, K. Ramchandran, C. Tomlin, ``Local differential privacy for multi-agent distributed optimal power flow,"  {\em IEEE PES Innovative Smart Grid Technologies Europe}, 2020. 


\bibitem{sundaram2007} S. Sundaram and C. Hadjicostis, ``Finite-time distributed consensus in
graphs with time-invariant topologies," {\em American Control Conference},  pp. 711-716, 2007.

\bibitem{yuan2013} Y. Yuan, G.-B. Stan, L. Shi, M. Barahona, and J. Goncalves, ``Decentralised
minimum-time consensus," {\em Automatica}, vol. 49, no. 5, pp. 1227-1235,
2013.

\bibitem{dwork2014algorithmic}
C.~Dwork, A.~Roth \emph{et~al.}, ``The algorithmic foundations of differential
  privacy,'' \emph{Foundations and Trends in Theoretical Computer Science},
  vol.~9, no. 3--4, pp. 211--407, 2014.

\bibitem{d1}
E.~Nozari, P.~Tallapragada, and J.~Cort{\'e}s, ``Differentially private average consensus: Obstructions, trade-offs, and optimal algorithm design",
  \emph{Automatica}, vol. 81, pp. 221--231, 2017.

\bibitem{d2}
J. He, L. Cai, and X. Guan, ``Differential private noise adding mechanism and its application on consensus algorithm",
  \emph{IEEE Transactions on Signal Processing}, vol. 68, pp. 4069--4082, 2020.

\bibitem{d3}
S.~Han, U.~Topcu, and G.~J. Pappas, ``Differentially private distributed
  constrained optimization,'' \emph{IEEE Transactions on Automatic Control},
  vol.~62, no.~1, pp. 50--64, 2016.

\bibitem{Campbell2009}
S. L. Campbell, and C. D. Meyer, \emph{Generalized inverses of linear
transformations}. SIAM, 2009.

\bibitem{siam1992} B. Kalyanasundaram and G. Schintger, ``The Probabilistic Communication Complexity of Set Intersection," {\em SIAM Journal on Discrete Mathematics},   Vol. 5, No. 4,  pp. 545-557, 1992.

\bibitem{focs} M. Braverman, F. Ellen, R. Oshman, T. Pitassi and V. Vaikuntanathan, ``Tight bounds for set disjointness in the message passing model," in {\em  54th Annual IEEE Symposium on. Foundations of Computer Science (FOCS)}, 2013.

\bibitem{peleg}D. Peleg. Distributed Computing: A Locality-Sensitive Approach.  Society for Industrial and Applied Mathematics, 2000.

\bibitem{PPSC} L. Wang, Y. Liu, I. R. Manchester, and G. Shi, ``Differentially Private Distributed Computation via Public-Private Communication Networks," Preprint arXiv:2101.01376, 2021.
\bibitem{cdc} H. Qi, B. Li, R.-J. Jing, A. Proutiere, and G, Shi, ``Distributedly solving Boolean equations over networks," in {\em The 59th IEEE Conference on Decision and Control}, 2020.
  \end{thebibliography}
\end{document}